\documentclass{article}

\usepackage[a4paper]{geometry}
\usepackage{xcolor}

\usepackage{fullpage} %Package pour parametrer la taille des marges
\usepackage{authblk}
\usepackage{amsmath}
\usepackage{amsthm}
\usepackage{amssymb}
\usepackage{esint}
\usepackage{mathrsfs}
\usepackage{listings}
\usepackage{mathtools}
\usepackage{caption}
\usepackage{subfigure}
\usepackage{graphicx}
\usepackage{hyperref}
\usepackage{tikz-cd}
\usepackage{standalone}
\usepackage{float}
\usetikzlibrary{patterns}

\usepackage[capitalise]{cleveref}
\let\C\relax % we want to keep \C free

\hypersetup{linktocpage}
\usepackage [english]{babel}
\usetikzlibrary{babel}
\usepackage [autostyle, english = american]{csquotes}
\MakeOuterQuote{"}

\newtheorem{theorem}{Theorem}[section]
\newtheorem{proposition}[theorem]{Proposition}
\newtheorem{definition}[theorem]{Definition}
\newtheorem{remark}[theorem]{Remark}
\newtheorem{lemma}[theorem]{Lemma}
\newtheorem{corollary}[theorem]{Corollary}

\newcommand{\R}{\mathbb{R}}
\newcommand{\N}{\mathbb{N}}
\newcommand{\Z}{\mathbb{Z}}
\newcommand{\Q}{\mathbb{Q}}
\newcommand{\T}{\mathbb{T}}
\newcommand{\C}{\mathbb{C}}
\newcommand{\Id}{\mathrm{I}}

\newcommand{\curl}{\mathrm{curl}\,}
\renewcommand{\div}{\mathrm{div}\,}

\usepackage{tocbasic} % For fixme
\usepackage[multiuser,draft,inline,marginclue]{fixme}
\fxusetheme{colorsig}
\FXRegisterAuthor{rou}{arou}{Roussel}% c'est pour toi
\FXRegisterAuthor{rr}{arr}{Remi}% c'est moi ca :P

\title{Shape differentiation for Poincaré maps of harmonic fields in toroidal domains}

\author[1]{Robin Roussel \thanks{robin.roussel@sorbonne-universite.fr}}

\affil[1]{Laboratoire Jacques-Louis Lions, Sorbonne Université, Inria, Paris}

\usepackage{hyperref}
\hypersetup{linktocpage}
\usepackage[capitalise]{cleveref}

\begin{document}
\maketitle
\begin{abstract}
In this article, we study Poincaré maps of harmonic fields in toroidal domains using a shape variational approach. Given a bounded domain of $\R^3$, we define its harmonic fields as the set of magnetic fields which are curl free and tangent to the boundary. For toroidal domains, this space is one dimensional, and one may thus single out a harmonic field by specifying a degree of freedom, such as the circulation along a toroidal loop. We are then interested in the Poincaré maps of such fields restricted to the boundary, which produce diffeomorphisms of the circle. We begin by proving a general shape differentiability result of such Poincaré maps in the smooth category, and obtain a general formula for the shape derivative. We then investigate two specific examples of interest; axisymmetric domains, and domains for which the harmonic field has a diophantine rotation number on the boundary. We prove that, in the first case, the shape derivative of the Poincaré map is always identically zero, whereas in the second case, assuming an additional condition on the geometry of the domain, the shape derivative of the Poincaré map may be any smooth function of the circle by choosing an appropriate perturbation of the domain. 
\end{abstract}
\tableofcontents

\section*{Introduction}\label{sec:intro}

When designing external magnetic fields for confinement in fusion reactors, such as tokamaks or stellarators, dynamical properties of the field lines play a key role in the stability properties of the plasma. In this context, spacial variations of the magnetic field magnitude lead to a drift of charged particles. However, introducing a twist in the magnetic field lines is known to average out the drift along the trajectories of charged particles \cite{imbert-gerard_introduction_2020}[Chapter 5]\cite{littlejohn_variational_1983,helander_collisional_2005}[Chapter 7]. To quantify this notion of twist, an important object in reactor design is the so-called rotational transform \cite{imbert-gerard_introduction_2020}[Chapter 7]. Assuming the magnetic field is foliated by two-dimensional tori, the rotational transform is defined on each leaf by the average ratio of poloidal turns and toroidal turns along the field lines. Mathematically, this is given by the rotation number of the Poincaré map on each leaf. 

In addition to giving information on the stability of charged particles in the plasma, the rotational transform is also useful in studying the topological and dynamical stability of the magnetic field itself. Indeed, due to errors arising from the coil design and fluctuations in the plasma, the actual magnetic field inside a reactor will be a perturbation of the theoretical one. Loosely speaking, KAM theory and Hamiltonian representations of magnetic fields lead to non-degenerate leaves of the magnetic field with a diophantine rotational transform being preserved after perturbations, whereas leaves with rational rotational transform may lead to chaotic regions and magnetic islands \cite{imbert-gerard_introduction_2020}[Chapter 10]\cite{lee_magnetic_1990}.

Mathematically, however, the assumption of a foliated magnetic field leads to complications. Indeed, the existence of such foliated magnetic fields is still closely linked to open questions. The most notable problem related to existence of foliations is Grad's conjecture \cite{grad_toroidal_1967}, which states that foliated smooth MHD equilibria with non-constant pressure should be axisymmetric. Theoretical results as well as a solid mathematical framework are therefore scarce when it comes to the study of rotational transform profiles. We still refer to a series of articles by Enciso, Luque and Peralta-Salas \cite{enciso_existence_2015,enciso_beltrami_2020,enciso_mhd_2022}, which study the dynamical properties of Beltrami fields. These articles develop a thorough theory to study Poincaré maps of Beltrami fields with small eigenvalue in thin toroidal tubes, and deduce several interesting results from this, such as the construction of non-trivial stepped pressure MHD equilibria in \cite{enciso_mhd_2022}.

In this paper, we are more specifically interested in the study of harmonic fields. Given a domain $\Omega$ of $\R^3$, we say that a vector field on $\Omega$ is harmonic if it is divergence free, curl free, and tangent to the boundary. When $\Omega$ has the topology of a full torus, the space of harmonic fields is one dimensional, and we may therefore single out a generator of this space by picking a normalization criterion, such as the circulation along a toroidal loop. From a physical point of view, harmonic fields are important in the design of stellarators, which aim to stabilize plasma without inducing current inside it. This therefore leads to magnetic fields with small curl inside the plasma domain, which may be approximated by harmonic fields. The fact that one may assign a harmonic field to each toroidal domain can lead to stellerator design using shape optimization techniques. This is what was done for example in \cite{robin_shape_2024} to optimize magnetic helicity of harmonic fields, which is another way to quantify the twist of a magnetic field from a topological point of view instead of a dynamical one. 

Although, to the author's knowledge, there is no clearly established conjecture in this direction, there seems to be no result on the existence of non-trivial foliated harmonic fields. To simplify things, we therefore choose to study the Poincaré maps of harmonic fields on the boundary only. Indeed, since harmonic fields of $\Omega$ are by definition tangent to the boundary, they define a flow on $\partial \Omega$. Therefore, if $\Omega$ is a toroidal domain, the Poincaré map of the harmonic field restricted to the boundary is a circle diffeomorphism, to which we may associate a rotation number. Since one may assign a harmonic field to each toroidal domain, the approach of this article is to investigate properties of the Poincaré maps of harmonic fields on the boundary using a shape differentiation approach, that is, to study how variations of the domain may lead to variations of the Poincaré map in the space of diffeomoprhisms of the circle. To avoid technicalities related to regularity, we choose to work in the smooth category throughout the article. We will therefore only be working with smooth domains, use smooth functions and vector fields, and prove smoothness of the studied objects when needed.

\subsection*{General approach and main results}\label{sub_sec:results}
Before discussing the contributions of the article, we give a formal introduction to the main objects we will study. The precise definitions will be given in \cref{sec:definitions}. Let $\Omega$ be a smooth toroidal domain, that is, a smooth open set of $\R^3$ such that $\bar{\Omega}$ is diffeomorphic to the full torus $S^1 \times D^2$, where $D^2$ is the closed unit disk of $\R^2$, and $S^1 = \R/\Z$. Given a curve $\gamma$ which generates the first homology group of $\bar{\Omega}$, there exists a unique harmonic field $B(\Omega)$ verifying 
$$
\int_\gamma B(\Omega) \cdot dl = 1.
$$
As will be further explained in \cref{sec:definitions}, this can bee seen by identifying harmonic fields as representatives of the first De Rham cohomology space of $\Omega$, and using De Rham's theorem. In order to define the Poincaré map of $B(\Omega)$ on the boundary as a diffeomorphism of the circle, we need the following data:
\begin{itemize}
    \item $\gamma$, a generator of the first homology group of $\bar{\Omega}$,
    \item $\Sigma$, a poloidal cut of $\partial \Omega$,
    \item Coordinates on $\Sigma$, that is, a diffeomorphism between $S^1$ and $\Sigma$.
\end{itemize}
Furthermore, $\Sigma$ needs to be a Poincaré cut of $B(\Omega)_{|\partial\Omega}$. All the required data and assumptions will be given by the notion of admissible embeddings of the torus in $\R^3$, which we denote $\mathrm{Emb}_{\mathrm{ad}}\left(\T^2;\R^3\right)$ (see \cref{def:admissible_embeddings}). We are thus able to consider the mapping 
$$
\Pi : \mathrm{Emb}_{\mathrm{ad}}\left(\T^2; \R^3\right) \rightarrow \mathrm{Diff}(S^1),
$$
which associates with each admissible embedding the Poincaré map of $B(\Omega)$, where $\Omega$ is the smooth toroidal domain whose boundary is the image of the embedding. We also note that we model $S^1$ as $\R / \Z$ throughout the paper, so that $S^1$ is equipped with canonical coordinates inherited from $\R$.

\begin{comment}
{\color{blue}\textbf{A CHANGER} An important object in the study of circle diffeomorphisms, which we will use mostly in \cref{sub_sec:rot_numbers}, is the rotation number. Let $f$ be an order preserving diffeomorphism of $S^1$, and $F$ a lift of $f$. We then define the rotation number of $f$ as 
$$
\rho(f) = \lim_{n \rightarrow +\infty}\frac{F^n(x)-x}{n} \hspace{1ex} (\text{mod } 1).
$$
This quantity is known to be well-defined and independent of the choice of initial point $x$ and lift $F$, and is at the center of the classification of circle homeomorphisms and diffeomorphisms. We refer the reader to \cite{katok_introduction_1995}[Chapter 11-12] for an introduction to the subject. We recall two classical results on this subject which will be used in the paper. First, if a smooth diffeomorphism has a diophantine rotation number, it is smoothly conjugate to a pure rotation \cite{yoccoz_analytic_2002}[Theorem 1.1], that is, a diffeomoprhism of the circle of the form
$$
F(x) = x + \omega,
$$
with $\omega$ in $S^1$. Second, there is an open dense set of diffeomorphisms of the circle which have rational rotation number \cite{devaney_introduction_2021}[Chapter 17]. The usual notion of such diffeomorphisms are Morse--Smale diffeomorhisms, but we will be using a larger class with similar properties which we will refer to as rationally stable diffeomorphisms (see \cref{def:stable}).}
\end{comment}

Although $\Pi$ is not, strictly speaking, a shape function (as it also depends on the coordinates on the boundary), the techniques we will use to study it are largely inspired by shape differentiation. Let $\mathcal{E}$ be an admissible embedding, and $t \mapsto P_{t}$ a differentiable path of smooth diffeomorphisms of $\R^3$ with $P_{0}=id$. Let $V$ in $\mathrm{Vec}\left(\R^3\right)$ be the derivative of $t \mapsto P_{t}$ at time $t=0$. As will be further explained later in the article, $\mathcal{E}_t := P_{t} \circ \mathcal{E}$ is then admissible as well for $t$ small enough. Our goal is then to study the derivative of $t \mapsto \Pi(\mathcal{E}_t)$ in the space of circle diffeomorphisms. More precisely, we will show that there exist a linear map $V \mapsto \Pi'(\mathcal{E}; V)$ such that 
$$
\Pi(\mathcal{E}_t) = \Pi(\mathcal{E}) + t\Pi'(\mathcal{E}; V) + o(t).
$$
We refer to $\Pi'(\mathcal{E}; V)$ as the shape derivative of the Poincaré map at $\mathcal{E}$ in the direction $V$. 

In this article, we establish this shape differentiability result and we study the image of the map $V \mapsto \Pi'(\mathcal{E}; V)$ in specific cases. The first case we will study is the one where $\mathcal{E}$ is the usual embedding of the standard axisymmetric torus. In this case, we will show that $\Pi'(\mathcal{E};V)$ actually vanishes for all $V$. This result is given in \cref{th:zero_derivative_axisymmetric}. Then, we will study the case where $\Pi(\mathcal{E})$ is a diophantine rotation. In this case, we will show that under an additional assumption on the geometry of the domain, the mapping $V \mapsto \Pi'(\mathcal{E};V)$ is surjective. This result is given in \cref{th:surjectivity_diophantine}. 
\begin{comment}
Finally, we prove in \cref{sec:thin_rot} that the set of admissible embeddings verifying the hypotheses of \cref{th:surjectivity_diophantine} is nonempty. This is done using thin enlargements of closed simple curves, for which we are able to use results from \cite{enciso_existence_2015,enciso_beltrami_2020,enciso_mhd_2022} on properties of harmonic fields in such domains.
\end{comment}

\subsection*{Outline of the article}
The article is organized as follows.
\begin{itemize}
    \item In \cref{sec:definitions}, we give a proper definition of the objects we will study throughout the paper. Firstly, we define in \cref{sub_sec:harmonic_fields} a way to associate a harmonic field with each toroidal domain. We then give a {weak} formulation for this harmonic field, which will be used during the shape differentiation process. In \cref{sub_sec:poincare_map}, we then define the notion of Poincaré map we will be studying. This is done first by defining a notion of admissible embeddings of the torus in $\R^3$ which provides the necessary data, and then by describing how we construct the Poincaré map from an admissible embedding.

    \item In \cref{sec:shape_differentiation}, we study the general shape differentiability of the Poincaré map of harmonic fields. The more demanding step is to prove shape differentiability of the harmonic fields in the smooth category, which is given by \cref{th:differentiability_B} of \cref{sub_sec:shape_diff_harmonic_fields}. The classical method to obtain Lagrangian shape differentiability of PDE solutions is to pull the {weak} formulation back onto a fixed domain, and to use an implicit function theorem argument (see for example \cite{henrot_shape_2018}[Chapter 5]). This approach was already taken to study the shape differentiability of harmonic fields in \cite{robin_shape_2024}. However, this method generally leads to shape differentiability in the variational space of the PDE, which in our case is H-curl. Since we want to differentiate the flow of this vector field, this regularity is not sufficient. We therefore proceed by identifying the correct shape derivative, and then estimate the associated first-order remainder in $\mathcal{C}^k$ norms using elliptic regularity results to obtain shape differentiability in the smooth category. In \cref{sub_sec:shape_diff_poinc}, we then obtain the shape derivative of the Poincaré map, which is relatively straightforward using the results of the previous section. We also provide a useful formula for the case in which the coordinates on the boundary linearize the harmonic field, which will be used in \cref{sec:axisymmetric,sec:diophantine}.
    
    \item In \cref{sec:axisymmetric}, we study the particular case of a standard axisymmetric torus. For axisymmetric domains the harmonic field is explicitly known, greatly simplifying the computations. We prove in \cref{th:zero_derivative_axisymmetric} that, in this case, the shape derivative of the Poincaré map always vanishes. This implies that around these domains, it is necessary to go to second-order in order to find local information about the Poincaré map of harmonic fields. The geometry of the domain plays a role in two steps of the proof. First through the explicit expression of the harmonic field and its relation with the curvature of the boundary, and second, through symmetries of the solution to a PDE which appears in the expression of the shape derivative of the harmonic field.
    \item In \cref{sec:diophantine}, we study the case where the Poincaré map has diophantine rotation number. Under an additional assumption relating the curvature of the boundary and the harmonic field, we prove \cref{th:surjectivity_diophantine}, which states that the shape derivative of the Poincaré map can be any smooth function of the circle if we choose a correct perturbation of the embedding. For this, we use cohomological equations to prove that the shape derivative of the Poincaré map can be any zero average function of the circle, and a specific normal perturbation to generate the last remaining dimension.
    \begin{comment}
    \item In \cref{sec:thin_rot}, we prove that the set of admissible embeddings satisfying the conditions of \cref{th:surjectivity_diophantine} is nonempty. To do this, we use results from \cite{enciso_existence_2015,enciso_beltrami_2020,enciso_mhd_2022} for harmonic fields in thin toroidal domains. The diophantine condition on the rotation number is a simple adaptation from a result of \cite{enciso_existence_2015}, and we obtain the second condition of \cref{th:surjectivity_diophantine} for thin enlargements of curves with nowhere vanishing torsion through a computation of the second fundamental form in such domains.
    \end{comment}
\end{itemize}

\subsection*{Notations}
\begin{itemize}
    \item $S^1_\ell = \R /(\ell \Z)$ and $S^1 = S^1_1$. We also define $\T^2 = \R^2 / \Z^2 \cong S^1 \times S^1$ and denote the closed unit disk of $\R^2$ as $D^2$.
    
    \item For two vectors $u$ and $v$ in $\R^3$, $u \cdot v$ is their Euclidean scalar product.

    \item Given a smooth manifold $M$ with (possibly empty) smooth boundary and $k$ in $\N \cup \{\infty\}$, $\mathcal{C}^k(M)$ is the space of real valued $k$ times differentiable functions on $M$, and $\mathrm{Vec}(M)$ is the set of smooth vector fields of $M$.
    
    \item Given a smooth manifold $M$ and a continuous family of vector fields $s \in \R \mapsto X_s \in \mathrm{Vec}(M)$, we denote 
    $$
    \overrightarrow{\exp}\int_{0}^{t} X_s ds,
    $$
    as the flow of $s \mapsto X_s$ at time $t$ when it is well-defined. In our case, the manifold will always be compact without boundary, so that there is global existence of flow.
    
    \item Let $X,Y$ be topological spaces and $f: X \rightarrow Y$ a continuous function. For $k$ in $\N$, $H_k(X)$ is the $k$-th singular homology group of $X$, and $f_* : H_k(X) \rightarrow H_k(Y)$ is the group morphism associated to $f$. We refer to \cite{hatcher_algebraic_2002}[Chapter 2] for the precise definitions of these objects. We note however that only basic homological notions will be used so that an intuitive understanding of singular homology and its relation with De Rham cohomology will be sufficient to understand its use in the paper.
    
    \item Suppose $\Omega$ is a smooth toroidal domain of $\R^3$, that is an open set such that $\bar{\Omega}$ is smoothly diffeomorphic to $S^1 \times D^2$, and $(\phi, \theta):\partial \Omega \rightarrow \T^2$ are smooth coordinates on $\partial \Omega$. 
    \begin{itemize}
        \item $n$ is the unit normal outward pointing vector field on $\partial \Omega$.

        \item $\div_\Gamma$ is the divergence on $\partial \Omega$, and $\nabla_\Gamma$ the tangential gradient. Both are defined using the metric on $\partial \Omega$ inherited from the Euclidean metric on $\R^3$.
    
        \item $\sqrt{g}$ is the square root of the determinant of the metric matrix in the $(\phi, \theta)$ coordinates. As such, the surface form on $\partial \Omega$ is given by $\sqrt{g} d\phi d\theta$.
        
        \item Given $\vec \omega = (\omega_1, \omega_2)$ in $\R^2$ and $f$ in $\mathcal{C}^\infty(\partial \Omega)$, we denote $\left\langle \vec \omega, \nabla_{\T^2} f \right\rangle = \omega_1 \partial_\phi f + \omega_2 \partial_\theta f$.
        
        \item Given a tangent vector $u$ on $\partial \Omega$, $u^\perp:=n \times u$.
    \end{itemize} 
\end{itemize}

\section{Definitions}\label{sec:definitions}

\subsection{Harmonic fields}\label{sub_sec:harmonic_fields}
Let $\Omega$ be a smooth toroidal domain of $\R^3$, that is, an open set of $\R^3$ such that $\bar{\Omega}$ is smoothly diffeomorphic to $S^1 \times D^2$. We define the space of harmonic fields of $\Omega$ as follows
$$
\mathcal{K}(\Omega) = \left\{{u} \in L^2(\Omega)^3 \mid \curl {u} = 0, \hspace{1ex} \div {u} = 0 \text{ and } {u}\cdot n = 0 \right\},
$$
where the curl and divergence should be understood in the weak sense. We now explain how one may single out a harmonic field in $\mathcal{K}(\Omega)$. Using the classical identification between vector fields and differential one-forms, we can relate the set of harmonic vector fields $\mathcal{K}(\Omega)$ to the set of harmonic one forms on $\bar{\Omega}$. Furthermore, from a classical result of Hodge theory (see \cite{schwarz_hodge_1995}[Theorem 2.6.1]) harmonic one forms are representatives of the first De Rham cohomology spaces of $\bar{\Omega}$. From this, we deduce that $\mathcal{K}(\Omega)$ is one dimensional. Then, choosing a generator $\gamma$ of the singular homology group $H_1\left(\bar{\Omega}\right) \cong \Z$, we know from De Rham's theorem that there exists a unique harmonic field $B(\Omega) \in \mathcal{K}(\Omega)$ such that 
$$
\int_\gamma B(\Omega) \cdot dl = 1.
$$
This harmonic vector field in fact also depends on the choice of generator $\gamma$, so that $B(\Omega)$ is a slight abuse of notation. However, if we were to choose a different generator $\tilde{\gamma} = \pm \gamma$, we would have 
$$
\int_{\tilde{\gamma}} B(\Omega) \cdot dl = \pm 1,
$$
so that changing the generator of $H_1(\bar{\Omega})$ can only change the harmonic field $B(\Omega)$ by a sign. We also note that, using the previously mentioned identification with harmonic one forms, we know from \cite{schwarz_hodge_1995}[Theorem 2.2.7] that $B(\Omega)$ is in fact smooth up to the boundary, that is, it is in $\mathrm{Vec}\left(\bar{\Omega}\right)$.

Although this definition is sufficient to characterize $B(\Omega)$, it will also be useful for shape differentiation to have a {weak} formulation for the harmonic field. Since $\Omega$ is a smooth toroidal domain, there exists a smooth embedding $\mathcal{F}: S^1 \times D^2 \ni {(\phi, x)}\mapsto \mathcal{F}{(\phi, x)} \in \bar{\Omega}$. We define the cutting surface $\Sigma$ of $\Omega$ as 
$$
\Sigma = \left\{ \mathcal{F}{(0, x)} \mid x \in D^2 \right\}.
$$
Therefore, $\mathcal{F}$ defines a diffeomorphism from ${(0,1) \times S^1}$ to $\bar{\Omega} \backslash \Sigma$. $\Omega \backslash \Sigma$ is then a simply connected pseudo-Lipschitz domain \cite{amrouche_vector_1998}[Definition 3.1]. Given a function $u$ in $H^1(\Omega \backslash \Sigma)$, $u \circ \mathcal{F}$ is in $H^1({(0,1) \times D^2})$, and we can define its traces on ${\{0\} \times D^2}$ and ${\{1\} \times D^2}$. This allows us to define the jump of $u$ across $\Sigma$ as 
$$
[\![u]\!]_{\Sigma} = \left(\left(u \circ \mathcal{F}\right)_{|{\{1\} \times D^2}}\right) \circ \mathcal{F}^{-1} - \left(\left(u \circ \mathcal{F}\right)_{|{\{0\} \times D^2}}\right) \circ \mathcal{F}^{-1},
$$
which is a function of $H^{1/2}(\Sigma)$. We now define for $c \in \R$
$$
V_c(\Omega \backslash \Sigma) = \left\{u \in H^1(\Omega \backslash \Sigma) \mid [\![u]\!]_{\Sigma} = c \right\}.
$$
From \cite{amrouche_vector_1998}[Lemma 3.11], we know that for $u$ in $V_c(\Omega \backslash \Sigma)$, $\nabla u$ extends to a curl free vector field of $\Omega$, which we denote $\tilde{\nabla} u$. There is also a natural identification between $V_0(\Omega \backslash \Sigma)$ and $H^1(\Omega)$. This allows us to construct the harmonic field in the following way, as is done for example in \cite{amrouche_vector_1998,alonso-rodriguez_finite_2018}.
\begin{proposition}\label{prop:construction_B}
    There exists a unique zero average solution to the following {weak formulation}. Find $u \in V_1(\Omega \backslash \Sigma)$ such that for all $v \in H^1(\Omega)$
    \begin{equation}\label{eq:B}
        \int_{\Omega} \nabla u \cdot \nabla v = 0.
    \end{equation}
    Furthermore, $\tilde{\nabla}u$ is a harmonic field of $\Omega$.
\end{proposition}

Moreover, the jump condition across $\Sigma$ leads to the equality 
$$
\int_{\mathcal{F}_*\gamma} \tilde{\nabla}u \cdot dl = 1,
$$
where $u$ is given by \cref{prop:construction_B}, $\gamma$ is the canonical generator of $H_1\left(S^1 \times D^2\right)$ and $\mathcal{F}_*$ is the isomorphism between $H_1\left(S^1 \times D^2\right)$ and $H_1\left(\bar{\Omega}\right)$ associated to $\mathcal{F}$. Therefore, $\tilde{\nabla}u$ is the harmonic field $B(\Omega)$ associated to the generator $\mathcal{F}_* \gamma$ of $H_1\left(\bar{\Omega}\right)$.

{
\begin{remark}
    This jump condition on $u$ arises from the fact that, using Poincaré's lemma, $B(\Omega)$ is locally a gradient vector field, but this is not true globally. We note that one may also use a multivalued, or $S^1$-valued, function $u$ to circumvent this jump condition. However, as mentioned earlier, this jump condition approach is already common for defining weak formulations of harmonic fields, and allows us to work in usual Sobolev spaces.
\end{remark}}

\begin{comment}
Let $\gamma$ be the canonical generator of the singular homology group $H_1\left(S^1 \times D^2\right)$. We also define the corresponding curve $\tilde{\gamma}$ in $(0,1)\times D^2$, obtained by removing the point at $\phi=0$ on $\gamma$. We thus obtain
\begin{align*}
    \int_{\mathcal{F}_* \gamma} \tilde{\nabla} u \cdot dl &= \int_{\mathcal{F}_* \tilde{\gamma}} du 
    \\
    &= \int_{\tilde{\gamma}} d\left(u \circ \mathcal{F}\right)
    \\
    &= \int_{\partial \tilde{\gamma}} u \circ \mathcal{F}
    \\
    &=1.
\end{align*}
\end{comment}
\subsection{Poincaré map}\label{sub_sec:poincare_map}
We now wish to define the Poincaré maps of harmonic fields on the boundary of toroidal domains. In order to do so, we proceed by specifying coordinates on the boundary. Indeed, having such coordinates allow{s} us to define a Poincaré cut and coordinates on this Poincaré cut, which is the required data to obtain the Poincaré map as a diffeomorphism of $S^1$. This is done by working with the set of smooth embeddings of $\T^2$ into $\R^3$, which we denote by $\mathrm{Emb}\left(\T^2;\R^3\right)$. We recall that, with an element $\mathcal{E}$ of $\mathrm{Emb}\left(\T^2; \R^3\right)$, we can associate an isomorphism $\mathcal{E}_*$ between the singular homology groups of $\T^2$ and the ones of $\mathcal{E}(\T^2)$. We denote by $\gamma_\phi$ and $\gamma_\theta$ the canonical generators of the homology group $H_1(\T^2)$. In order to be able to define the Poincaré map, we need to make some further assumptions on the embedding which are given by the following definition.

\begin{definition}\label{def:admissible_embeddings}
    Let $\mathcal{E}$ be in $\mathrm{Emb}\left(\T^2;\R^3\right)$. We say $\mathcal{E}$ is admissible if it satisfies the following conditions.
    \begin{itemize}
        \item $\mathcal{E}(\T^2)$ bounds a smooth toroidal domain $\Omega$ in $\R^3$.
        \item $\mathcal{E}$ is toroidal, that is 
        \begin{itemize}
            \item $\mathcal{E}_* \gamma_\phi$ is trivial in $H_1\left(\Omega^c\right)$ and generates $H_1\left(\bar{\Omega}\right)$,
            \item $\mathcal{E}_* \gamma_\theta$ is trivial in $H_1\left(\bar{\Omega}\right)$ and generates $H_1\left(\Omega^c\right)$.
        \end{itemize}
        \item $\mathcal{E}$ is transverse, that is, if $B(\Omega)$ is the harmonic field of $\Omega$ associated to the generator $\mathcal{E}_* \gamma_\phi$ then, $B(\Omega)^\phi$ is positive on $\partial \Omega$, where $(\phi, \theta) = \mathcal{E}^{-1}$ are the coordinates induced on $\partial \Omega$ by $\mathcal{E}$.
    \end{itemize} 
    We denote by $\mathrm{Emb}_{\mathrm{ad}}\left(\T^2; \R^3\right)$ the set of admissible embeddings of $\T^2$ into $\R^3$.
\end{definition}

\begin{remark}
    Here are a few remarks which may help the reader to interpret the definition of admissible embeddings:
    \begin{itemize}
        \item The first condition of \cref{def:admissible_embeddings} is not redundant. Indeed, although a smoothly embedded torus in $S^3$ always bounds a full torus, this result is not true for embeddings in $\R^3$. We refer to \cite{arnaud_recognition_2010}[Definition 3] for a description of such domains, referred to as knotted anti-toi, as well as \cite{cantarella_vector_2002}[Figure 13] for an illustration of such embedded tori.
        
        \item The second condition of \cref{def:admissible_embeddings} essentially states that $(\phi, \theta) = \mathcal{E}^{-1}$ define toroidal and poloidal coordinates respectively on the boundary of $\Omega$. Although it is only necessary to assume that $\mathcal{E}_* \gamma_\phi$ generates $H_1\left(\bar{\Omega}\right)$ to define the Poincaré map, the additional assumptions are here to ensure that the Poincaré map we will construct corresponds to what we may expect geometrically. For example, the assumption that $\mathcal{E}_* \gamma_\theta$ is trivial in $H_1\left(\bar{\Omega}\right)$ means that curves of constant $\phi$ correspond to poloidal cuts of $\partial \Omega$.
        
        \item If $\mathcal{F}:S^1 \times D^2 \rightarrow \R^3$ is a smooth embedding and $i: \T^2 \rightarrow S^1 \times D^2$ is the canonical injection onto $\partial \left(S^1 \times D^2\right)$, we obtain that $\mathcal{F} \circ i$ verifies the first two assumptions of \cref{def:admissible_embeddings} with $\Omega = \mathcal{F}\left(S^1 \times D^2\right)$.
        
        \item The last condition of \cref{def:admissible_embeddings} ensures that $B(\Omega)$ is transverse to poloidal cuts, that is, nowhere tangent to curves of constant $\phi$. Therefore, its Poincaré map may be defined on such cuts.
    \end{itemize}
\end{remark}

\begin{comment}
\begin{remark}
    Using the Mayer--Vietoris sequence in singular homology \cite{hatcher_algebraic_2002}[Chapter 2.2], we know that given an open domain $\Omega$ of $\R^3$, we have the short exact sequence
    \[\begin{tikzcd}
        {H_2(\R^3)} & {H_1(\partial \Omega)} & {H_1\left(\bar \Omega\right) \oplus H_1\left(\Omega^c\right)} & {H_1\left(\R^3\right)}
        \arrow[from=1-1, to=1-2]
        \arrow[from=1-2, to=1-3]
        \arrow[from=1-3, to=1-4]
    \end{tikzcd},\]
    so that $H_1(\partial \Omega)$ is isomorphic to $H_1\left(\bar \Omega\right) \oplus H_1\left(\Omega^c\right)$. Since $H_1(\T^2) \cong \Z^2$ and $H_1(S^1 \times D^2) \cong \Z$, we know that if $\mathcal{E}$ satisfies the first hypothesis of \cref{def:admissible_embeddings}, there exists $\Phi$ in $\mathrm{GL}(2;\Z)$ (which may be seen as a diffeomorphism of $\T^2$) such that $\mathcal{E} \circ \Phi$ satisfies the second hypothesis of \cref{def:admissible_embeddings}.
\end{remark}
\rouwarning{I don't know if it is worth giving more details on this remark.}
\end{comment}

We now explain how we can define the Poincaré maps of harmonic fields. Let $\mathcal{E}$ be an admissible embedding, $\Omega$ be the toroidal domain such that $\partial \Omega = \mathcal{E}\left(\T^2\right)$, $B(\Omega)$ the harmonic field of $\Omega$, and $(\phi, \theta)$ the coordinates on $\partial \Omega$ associated with $\mathcal{E}$. First, to define the Poincaré map, {it is useful} to normalize the harmonic field. This is done by defining the following vector field on $\partial \Omega$:
\begin{equation}\label{eq:def_X}
    X(\mathcal{E}) = \frac{B(\Omega)}{B(\Omega)^\phi}.
\end{equation}
From this definition, and the fact that $B(\Omega)^\phi$ is positive, we know that the field lines of $X(\mathcal{E})$ correspond to the ones of $B(\Omega)$ up to an order-preserving reparametrization of time. Furthermore, we get that $X(\mathcal{E})^\phi = 1$, so that the field lines of $X(\mathcal{E})$ evolve linearly in $\phi$. This implies that if a field line starts on the poloidal cut $\phi=0$ at time $t=0$, it will return to the same cut at time $t=1$, which is precisely what we need for the Poincaré map. 
\\
We may therefore define the Poincaré map of $B(\Omega)$ as the one time flow of $X(\mathcal{E})$ restricted to the cut $\phi = 0$. However, it is more convenient to work on the fixed space $S^1$ in order to study variations of the Poincaré map. This can be done once again using the $(\phi, \theta)$ coordinates associated with $\mathcal{E}$. Let $S^1 \ni \phi \mapsto X_\phi(\mathcal{E}) \in \mathrm{Vec}(S^1)$ be the one-parameter family of vector fields given by 
\begin{equation}\label{eq:def_X_phi}
    X_\phi(\mathcal{E})(\theta) = X(\mathcal{E})^\theta(\phi, \theta) e_\theta,
\end{equation}
where $e_\theta$ is the canonical unit vector field of $S^1$. We then define the Poincaré map $\Pi(\mathcal{E})$ as 
\begin{equation}\label{eq:def_Pi}
    \Pi(\mathcal{E}) = \overrightarrow{\mathrm{exp}} \int_{0}^{1} X_\phi(\mathcal{E}) d\phi,
\end{equation}
which is a diffeomorphism of the circle. It will also prove to be useful to define the same flow at time $\phi$, which we denote by $\Pi^\phi(\mathcal{E})$. 

\section{Shape differentiation}\label{sec:shape_differentiation}
In this section, we consider an admissible embedding $\mathcal{E}$, $\Omega$ its corresponding domain, and $t \mapsto P_t$ a differentiable family of diffeomoprhisms of $\R^3$ with $P_0=id$. We denote
$$
V := \frac{d}{dt}_{\big| t=0} P_t,
$$
which is a smooth vector field of $\R^3$. Denoting $\mathcal{E}_t = P_t \circ \mathcal{E}$, our goal is to prove that $t \mapsto \Pi(\mathcal{E}_t)$ is differentiable in $\mathrm{Diff}(S^1)$. More precisely, we will identify a linear map $V \mapsto \Pi'(\mathcal{E}; V)$ such that 
$$
\Pi(\mathcal{E}_t) = \Pi(\mathcal{E}) + t \Pi'(\mathcal{E}; V) + o(t).
$$

\subsection{Shape differentiation of harmonic fields in the smooth category}\label{sub_sec:shape_diff_harmonic_fields}

Before studying the shape differentiability of the Poincaré map, we need to prove that $t \mapsto B(\Omega_t)$ is itself shape differentiable. Here, $\Omega$ is the domain associated with $\mathcal{E}$ and $\Omega_t = P_t(\Omega)$ is the one associated to $\mathcal{E}_t$. 

The classical approach for such problems, that is, shape differentiability of solutions to PDEs, is to define a certain way to pullback the solutions onto the fixed domain $\Omega$, and to use an implicit function argument on the pulled-back {weak} formulation \cite{henrot_shape_2018}[Chapter 5]. However, this only leads to shape differentiability in the variational space of the PDE, which in the case of \cref{prop:construction_B} is $H(\curl,\Omega)$. One therefore needs to use elliptic regularity results to obtain shape differentiability in the smooth category. This is done for example in \cite{henrot_shape_2018}[Section 5.5] for a Poisson problem with Neumann boundary conditions using the {weak} formulation restricted to $H^k$ spaces and an implicit function argument. We however, will estimate the difference between the solution to our PDE and its first-order approximation with respect to the deformation directly in $H^k$ norms, which in the end uses similar results of elliptic regularity.

Before stating the main result of this section, we note that the way in which we choose to pullback the harmonic field onto the fixed domain $\Omega$ affects the final result for the shape derivative formula. We will use two ways of pulling back $B_t$ onto $\Omega$. The first one is obtained by taking the pushforward by $P_t^{-1}$ of $B_t$ as a vector field. This has the advantage of preserving the field lines of $B_t$, which is precisely what we want in order to study the change in dynamics of the harmonic fields. The second one is obtained by taking the pullback by $P_t$ of $B_t$ when the latter is seen as a one-form on $\Omega_t$. This is given by the transformation $(P_t)^*_1$ which we will introduce in \cref{def:pullbacks}. Although this transformation does not preserve field lines, we will see that it behaves well with respect to the {weak} formulation for harmonic fields given in \cref{prop:construction_B}. Moreover, this transformation maps gradient fields to gradient fields, and curl-free fields to curl-free fields\footnote{In the language of differential forms; exact forms to exact forms, and closed forms to closed forms}. This property will be used in order to reduce the shape differentiability problem to elliptic regularity estimates on a classical PDE with scalar-valued solutions. Although we will not be using it here, another natural way to transform the harmonic fields is to take the pullback by $P_t$ when $B_t$ is seen as a two-form. This is what was done in \cite{robin_shape_2024} to prove shape differentiability of the harmonic field in a less regular context. We refer the reader to \cite{hiptmair_shape_2013} for elements of shape differentiation using the differential forms formalism. 

Throughout this section, we will often decompose vector fields in the canonical Cartesian basis. Moreover, if $u$ and $v$ are vector fields in a domain of $\R^3$, $Du$ is the Jacobian matrix field of $u$ in Cartesian coordinates, and $(Du) v$ is the matrix vector product in Cartesian coordinates, whenever these objects are well defined.

\begin{theorem}\label{th:differentiability_B}
    The mapping
    $$
    \begin{cases}
        \R \rightarrow \mathrm{Vec}\left(\bar{\Omega}\right) \\
        t \mapsto \left(P_t^{-1}\right)_* B_t
    \end{cases}
    $$
    is differentiable at zero, and its derivative is given by
    \begin{equation}\label{eq:B'}
        B'_V = [V, B_0] + \nabla u_V,
    \end{equation}
    where $u_V \in \mathcal{C}^\infty(\Omega)$ verifies 
    \begin{equation}\label{eq:u}
        \begin{cases}
            \Delta u_V = 0 & \text{in } \Omega, 
            \\
            \nabla u_V \cdot n = \div_{\Gamma}\left(B_0 \left(V\cdot n\right)\right) & \text{on } \partial \Omega.
        \end{cases}
    \end{equation}
\end{theorem}

\begin{comment}
In order to study the differentiation of $t \mapsto B_t$ in the smooth category, we first need to see why the harmonic fields are smooth in the first place. This comes from identifying $B_t$ as the 1-form $\beta_t := B_t^\flat$, where $\flat$ is the musical isomorphism relating vector fields to 1-forms using the Eucledian metric on $\Omega_t$. Indeed, from this we have 
\begin{align*}
    &(*d\beta_t)^\flat = \curl B_t = 0,
    \\
    &\delta \beta_t = -\div B_t = 0,
\end{align*}
where $*$ is the Hodge star operator relating 2-forms to 1-forms. Therefore, $\beta_t$ is harmonic. Furthermore, given a vector field $X$ of $\partial \Omega_t$ of tangent part $X_\parallel$, the normal part of $\beta_t$ is given by 
\begin{align*}
    \mathrm{n}\beta_t(X) &= \beta_t(X) - \beta_t\left(X_\parallel\right)
    \\
    &= \beta_t((X \cdot n_t) n_t)
    \\
    &= (X \cdot n )(B_t \cdot n_t)
\end{align*}
Therefore, $\beta_t$ has zero normal part, and is a Neumann harmonic field. From \cite{schwarz_hodge_1995}[Theorem 2.2.7] and the fact that $\bar{\Omega}_t$ is a manifold with smooth boundary, we then deduce that $\beta_t$ is smooth.
\end{comment}

In order to prove \cref{th:differentiability_B}, we define ways to pullback functions and vector fields of $\Omega_t$ onto the fixed domain $\Omega$. Although these transformations are quite common for shape differentiation of classical boundary value problems, we choose to introduce them in a way which clarifies their link with pullbacks of differential forms. This is what was done for example in \cite{robin_shape_2024}.
\begin{definition}\label{def:pullbacks}
    Let $v_0$ be in $L^2(\Omega_t)$ and $v_1$ in $L^2(\Omega_t)^3$. We define 
    \begin{align*}
        (P_t)_0^* v_0 &= v_0 \circ P_t,
        \\
        (P_t)_1^* v_1 &= D P_t^T v_1 \circ P_t.
    \end{align*}
\end{definition}

The proof of \cref{th:differentiability_B} will come in two steps. First, we will prove that $t \mapsto (P_t)_1^* B_t$ is shape differentiable. As we will see, this way of pulling back $B_t$ onto $\Omega$ behaves well with respect to the {weak} formulation given in \cref{prop:construction_B} to construct the harmonic fields. Also, using the transformations of \cref{def:pullbacks} and the aforementioned {weak} formulation, we will be able to write the difference between $t \mapsto (P_t)_1^* B_t$ and its first-order expansion at $t=0$ as the gradient of a function $\varphi_t$. Furthermore, we will show that $\varphi_t$ satisfies a linear elliptic PDE, on which we will use classical elliptic regularity results in order to prove that $\varphi_t$ is $o(t)$ in $\mathcal{C}^k$ for all $k$. The rest of the proof will then come by composing the correct transformations in order to recover $(P_t^{-1})_* B_t$ from $(P_t)^*_1 B_t$, and differentiating.

Before studying the shape differentiation of $B_t$, we give some useful properties of the pullbacks introduced in \cref{def:pullbacks}. We will be using the objects introduced in \cref{sub_sec:harmonic_fields} to define the {weak} formulation in $\Omega$ with an additional $t$ in subscript for the corresponding objects in the domain $\Omega_t$. We recall that $H(\curl, \Omega)$ is the space of square integrable vector fields of $\Omega$ which have square integrable curl.
\begin{lemma}\label{lemma:cd_pullbacks}
    Let $c$ be a real number. Then, the diagram
    \begin{equation}\label{eq:cd}
        \begin{tikzcd}
            {V_c(\Omega_t \backslash \Sigma_t)} & {H(\curl, \Omega_t)}
            \\
            {V_c(\Omega \backslash \Sigma)} & {H(\curl, \Omega)}
            \arrow["\tilde{\nabla}", from=1-1, to=1-2]
            \arrow["\tilde{\nabla}", from=2-1, to=2-2]
            \arrow["(P_t)_0^*", from=1-1, to=2-1]
            \arrow["(P_t)_1^*", from=1-2, to=2-2]
        \end{tikzcd}
    \end{equation}
    is commutative.
\end{lemma}

\begin{proof}
    As was mentioned earlier, the horizontal arrows of (\ref{eq:cd}) are given by \cite{amrouche_vector_1998}[Lemma 3.11]. We thus begin by proving that the vertical arrows are well-defined. 
    For the first arrow, take $u$ in $H^1(\Omega_t \backslash \Sigma_t)$. It is then straightforward that $\Phi_{tV}^{0}u = u \circ P_t$ is in $H^1(\Omega \backslash \Sigma)$. Furthermore, we have using $\mathcal{F}_t = P_t \circ \mathcal{F}$
    \begin{align*}
        [\![(P_t)_0^* u]\!]_{\Sigma} &= \left(\left((P_t)_0^* u \circ \mathcal{E}\right)_{|D^2 \times \{1\}}\right) \circ \mathcal{E}^{-1} - \left(\left((P_t)_0^* u \circ \mathcal{E}\right)_{|D^2 \times \{0\}}\right) \circ \mathcal{E}^{-1}
        \\
        &= \left[ \left(\left(u \circ \mathcal{E}_t\right)_{|D^2 \times \{1\}}\right) \circ \mathcal{E}_t^{-1} - \left(\left(u \circ \mathcal{E}_t\right)_{|D^2 \times \{0\}}\right) \circ \mathcal{E}_t^{-1} \right] \circ P_t
        \\
        &= [\![u]\!]_{\Sigma_t} \circ P_t,
    \end{align*}
    so that if $u$ is in $V_c(\Omega_t \backslash \Sigma_t)$, $(P_t)_0^* u$ is in $V_c(\Omega \backslash \Sigma)$. To prove that $(P_t)_1^*$ maps $H(\curl, \Omega_t)$ to $H(\curl, \Omega)$ and that the diagram is commutative, the computations are exactly the same as in the proof of \cite{robin_shape_2024}[Proposition 4], so we do not give the details here.
\end{proof}

\begin{lemma}\label{lemma:alpha}
    Let $u$ and $v$ be in $L^2(\Omega_t)^3$. Then 
    $$
    \int_{\Omega_t} u \cdot v = \int_{\Omega} (\alpha(t)(P_t)_1^* u)\cdot((P_t)_1^* v),
    $$
    where $\alpha(t) = \det\left(DP_t\right)DP_t^{-1}DP_t^{-T}$. Furthermore $t \in \R \mapsto \alpha(t) \in \mathcal{C}^\infty\left(\bar{\Omega}, \mathcal{M}_3(\R)\right)$ is differentiable at zero, and its derivative verifies, for every $u$ in $\mathrm{Vec}\left(\bar{\Omega}\right)$, 
    $$
    \alpha'(0)u = (\div u) V + \curl(u \times V) - \curl u \times V - \nabla(V \cdot u).
    $$
\end{lemma}

\begin{proof}
    The first equality can be found by a simple change of variables, and algebraic manipulation. For the second statement, we first notice that $\R \ni t \mapsto DP_t \in \mathcal{C}^\infty\left(\bar{\Omega}, \mathcal{M}_3(\R)\right)$ is differentiable at zero, and its derivative is $DV$. Moreover, for $t=0$, $DP_t$ is constant on $\bar{\Omega}$ equal to the identity matrix. Also, $M \mapsto \det(M)$, $M \mapsto M^{-1}$ and $M \mapsto M^{-T}$ are differentiable at the identity. We deduce that $t \mapsto \alpha(t)$ is indeed differentiable at $0$. Its derivative is then given by
    \begin{align*}
        \alpha'(0) &= \mathrm{tr}\left(DV\right)\Id - DV - DV^T,
        \\
        &= \div(V)\Id - DV - DV^{T}.
    \end{align*}
    Given $u$ in $\mathrm{Vec}\left(\bar{\Omega}\right)$ and combining the identities 
    \begin{align*}
        \nabla(V \cdot u) &= DV^T u + Du^T V,
        \\
        \curl(u \times V) &= DuV - DVu - (\div u)V + (\div V)u,
        \\
        (Du - Du^T)V &= \curl u \times V,
    \end{align*}
    we find the desired formula.
\end{proof}

\begin{proposition}\label{prop:phi_t}
    Let $\varphi_t$ be defined as 
    $$
    \varphi_t = (P_t)_0^* u_t - u_0 - t(u_V + V\cdot B_0), \hspace{1ex} t \in \R
    $$
    Where $u_t$ is the solution to the {weak} formulation given in \cref{prop:construction_B} in the deformed domain $\Omega_t$. Then, $\varphi_t$ is a function in $H^1(\Omega)$ which solves 
    \begin{equation}\label{eq:PDE_phi_t}      
        \begin{cases}
            -\div(\alpha(t)\nabla \varphi_t) = \div\left(\alpha_1(t)B_0\right) + t\div\left(\alpha_0(t)\nabla(u_V + V \cdot B_0)\right) &\text{in } \Omega,
            \\
            (\alpha(t)\nabla \varphi_t) \cdot n = -\left(\alpha_1(t)B_0\right) \cdot n - t \left(\alpha_0(t)\nabla(u_V + V\cdot B_0)\right)\cdot n &\text{on } \partial \Omega,
        \end{cases}
    \end{equation}
    with $\alpha_0(t) = \alpha(t) - \Id$ and $\alpha_1(t) = \alpha(t) - \Id - t\alpha'(0)$.
\end{proposition}

\begin{proof}
    First, we know from the definition of $u_t$ and the commutativity of (\ref{eq:cd}) that $u_0$ and $(P_t)_0^*u_t$ are in $V_1(\Omega \backslash \Sigma)$. Therefore, $(P_t)_0^*u_t-u_0$ is in $V_0(\Omega \backslash \Sigma) \cong H^1(\Omega)$. We deduce that $\varphi_t$ is indeed in $H^1(\Omega)$. 
    
    Now, we identify a {weak} formulation for $(P_t)^*_0 u_t$ on the fixed domain $\Omega$. We recall that $u_t$ is a function of $V_1(\Omega_t \backslash \Sigma_t)$ such that for all $v$ in $H^1(\Omega_t)$
    $$
    \int_{\Omega_t} \nabla u_t \cdot \nabla v =0.
    $$
    Using (\ref{eq:cd}) and \cref{lemma:alpha}, we therefore find 
    $$
    \int_{\Omega} \left(\alpha(t)\left(\nabla (P_t)_0^* u_t\right)\right) \cdot \left(\nabla (P_t)_0^* v\right) = 0.
    $$
    Moreover, we can observe that $(P_t)_0^*$ is an isomorphism from $H^1(\Omega_t)$ to $H^1(\Omega)$ with inverse $(P_t^{-1})_0^*$. Therefore, for all $v$ in $H^1(\Omega)$, we have 
    \begin{equation}\label{eq:u_t_pullback}
        \int_{\Omega} \left(\alpha(t)\left(\nabla (P_t)_0^* u_t\right)\right) \cdot \nabla v = 0.
    \end{equation}
    Now, using $\varphi_t = (P_t)_0^* u_t - u_0 - t(u_V + V \cdot B_0)$, we have 
    \begin{align*}
        \int_{\Omega} \left(\alpha(t)\nabla \varphi_t\right) \cdot \nabla v =&  \int_{\Omega} \left(\alpha(t)\left(\nabla (P_t)_0^* u_t\right)\right) \cdot \nabla v - \int_{\Omega}\left(\alpha(t)\nabla u_0\right) \cdot \nabla v - t \int_{\Omega} \left(\alpha(t)\nabla\left(u_V + V \cdot B_0\right)\right) \cdot \nabla v
        \\
        =&- \int_{\Omega} \nabla u_0 \cdot \nabla v - t \int_{\Omega} \left(\alpha'(0)\nabla u_0\right) \cdot \nabla v - t\int_{\Omega} \nabla(u_V + V\cdot B_0)\cdot \nabla v
        \\
        & - \int_{\Omega} \left[\left(\alpha(t)-\Id-t\alpha'(0)\right) \nabla u_0 \right] \cdot \nabla v - t \int_{\Omega} \left[\left(\alpha(t)-\Id\right)\nabla(u_V + V \cdot B_0)\right] \cdot \nabla v,
    \end{align*}
    where we used \cref{eq:u_t_pullback}. Using $B_0 = \tilde{\nabla}u_0$ and the fact that $B_0$ is $L^2$ orthogonal to gradient vector fields, we obtain
    \begin{align*}
        \int_{\Omega} \left(\alpha(t)\nabla \varphi_t\right) \cdot \nabla v
        =& - t \int_{\Omega} \left(\alpha'(0)B_0\right) \cdot \nabla v - t\int_{\Omega} \nabla(u_V + V\cdot B_0)\cdot \nabla v
        \\
        & -\int_{\Omega} \left(\alpha_1(t) B_0 \right) \cdot \nabla v -t \int_{\Omega} \left(\alpha_0(t) \nabla(u_V + V \cdot B_0)\right) \cdot \nabla v.
    \end{align*}
    We now want to prove that the first two terms of the previous equation cancel out. Let us choose a smooth test function $v$ in $\mathcal{C}^\infty\left(\bar{\Omega}\right)$. Using the definition of $u_V$ given in \cref{eq:u}, we have 
    \begin{align*}
        \int_{\Omega} \nabla u_V \cdot \nabla v &= \int_{\partial \Omega} \div_\Gamma (B_0 (V \cdot n)) v
        \\
        &= -\int_{\partial \Omega} (B_0 \cdot \nabla_\Gamma v) V \cdot n 
        \\
        &= -\int_{\partial \Omega} (B_0 \cdot \nabla v) V \cdot n.
    \end{align*}
    Using \cref{lemma:alpha}, we also have 
    \begin{align*}
        \int_{\Omega}(\alpha'(0)B_0) \cdot \nabla v &= \int_{\Omega} (\div B_0)V \cdot \nabla v + \curl (B_0 \times V) \cdot \nabla v - (\curl B_0 \times V) \cdot \nabla v - \nabla (B_0 \cdot V) \cdot \nabla v
        \\
        &= \int_{\Omega} \curl(B_0 \times V) \cdot \nabla v - \nabla(B_0 \cdot V) \cdot \nabla v
        \\
        &= -\int_{\partial \Omega} \left[(B_0 \times V) \times n\right] \cdot \nabla v - \int_{\Omega}\nabla(B_0 \cdot V) \cdot \nabla v
        \\
        &= \int_{\partial \Omega} \left(B_0 \cdot \nabla v\right) V \cdot n - \int_{\Omega}\nabla(B_0 \cdot V) \cdot \nabla v,
    \end{align*}
    where we used the integration by parts formula for the curl, and the identity $(a \times b) \times c = (a \cdot {c}) {b} - (b \cdot c) a$. Therefore, 
    $$
    \int_{\Omega} \left(\alpha'(0)B_0\right) \cdot \nabla v + \int_{\Omega} \nabla(u_V + V \cdot B_0) \cdot \nabla v = 0, 
    $$
    which implies 
    $$
    \int_{\Omega} \left(\alpha(t)\nabla \varphi_t\right) \cdot \nabla v = -\int_{\Omega} \left(\alpha_1(t) B_0 \right) \cdot \nabla v -t \int_{\Omega} \left(\alpha_0(t)\nabla(u_V + V \cdot B_0)\right) \cdot \nabla v.
    $$
    Since $\mathcal{C}^\infty\left(\bar{\Omega}\right)$ is dense in $H^1(\Omega)$, the previous formula actually holds for any function $v$ in $H^1(\Omega)$. Finally, integrating by parts, we obtain 
    \begin{align*}
        \int_{\Omega} \left(\alpha(t)\nabla \varphi_t\right) \cdot \nabla v =& \int_{\Omega} \div(\alpha_1(t)B_0)v - \int_{\partial \Omega} \left[\left(\alpha_1(t)B_0\right) \cdot n \right] v + t\int_{\Omega}\div(\alpha_0(t)\nabla(u_V + V \cdot B_0)) v
        \\
        & - t\int_{\partial \Omega} \left[ \left(\alpha_0(t)\nabla(u_V + V\cdot B_0)\right) \cdot n \right] v,
    \end{align*}
    which is the {weak} formulation of the desired equation.
\end{proof}

Since $\alpha_0(t)$ and $\alpha_1(t)$ are {of order} $o(t)$ {as $t \rightarrow 0$}, we expect $\varphi_t$ to be as well. Classical results from elliptic regularity allows us to prove the following proposition.

\begin{proposition}\label{prop:ell_reg_phi_t}
    For all $k$ in $\N$, we have as $t$ goes to zero 
    $$
    \left\|\varphi_t - \fint_{\Omega} \varphi_t \right\|_{H^k(\Omega)} = o(t).
    $$
\end{proposition}

Before proving this proposition, we prove an immediate corollary.

\begin{corollary}\label{cor:diff_phi_tV_B_t}
    The mapping
    $$
    \begin{cases}
        \R \rightarrow \mathrm{Vec}\left(\bar{\Omega}\right)
        \\
        t \mapsto (P_t)_1^* B_t
    \end{cases}
    $$ 
    is differentiable at zero, and its derivative is given by 
    $$
    \frac{d}{dt}_{\big{|}t=0} (P_t)_1^* B_t = \nabla u_V + \nabla(V \cdot B_0).
    $$
\end{corollary}

\begin{proof}
    From \cref{prop:ell_reg_phi_t}, we have for each $k$ in $\N$
    $$
    \|\nabla \varphi_{t}\|_{H^k(\Omega)} = o(t).
    $$
    Now, using Sobolev injections, we deduce that for all $k$ in $\N$
    \begin{equation}\label{eq:C^k_estimate}
        \|\nabla \varphi_{t}\|_{\mathcal{C}^k\left(\bar{\Omega}\right)} = o(t).
    \end{equation}
    Finally, from the expression of $\varphi_t$ and the commutativity of diagram (\ref{eq:cd}), we get 
    $$
    \nabla \varphi_t = (P_t)_1^* B_t - B_0 - t(\nabla u_V + \nabla(V \cdot B_0)).
    $$
    We thus obtain the desired result from \cref{eq:C^k_estimate}.
\end{proof}

\begin{proof}[Proof of \cref{prop:ell_reg_phi_t}]
    Since $u_V$ solves 
    $$
    \begin{cases}
        \Delta u_V = 0 & \text{in } \Omega, 
        \\
        \nabla u_V \cdot n = \div_{\Gamma}\left(B_0 (V\cdot n)\right) & \text{on } \partial \Omega,
    \end{cases}
    $$
    and $B_0$ is smooth on $\partial \Omega$, we know that $u_V$ is smooth from elliptic regularity \cite{grisvard_elliptic_2011}[Theorem 2.5.1.1]. Thus, all the source terms (resp. boundary terms) of \cref{eq:PDE_phi_t} are smooth, and in particular are in $H^k(\Omega)$ (resp. $H^k(\partial \Omega)$) for all $k$. Now define 
    $$
    \tilde{\varphi}_t = \varphi_t - \fint_{\Omega} \varphi_t.
    $$
    For $\xi \in \R^3$ and $t$ small enough, we have 
    $$
    (\alpha(t)\xi) \cdot \xi = \det\left(DP_t\right) \left|DP_t^{{-T}}\xi\right|^2 \geq C |\xi|^2,
    $$
    where $C$ is positive and independent of $t$. Therefore, by Lax--Milgram, $\tilde{\varphi}_t$ is the unique zero average solution to \cref{eq:PDE_phi_t}. To shorten the notations, we define $\tilde{u}_V = u_V + V \cdot B_0$. From \cite{grisvard_elliptic_2011}[Section 2.5.1], we know that $\tilde{\varphi}_t$ is in $H^k(\Omega)$ for all $k$, and 
    \begin{align}\label{eq:ineq_regularity}
        \begin{split}
            \|\tilde{\varphi}_t\|_{H^{k+2}(\Omega)} \leq C_{k, t} &\left( \left\| \div(\alpha_1(t)B_0) \right\|_{H^k(\Omega)} + \left\|t \div(\alpha_0(t)\nabla \tilde{u}_V)\right\|_{H^k(\Omega)} + \right.
            \\
            &\left. \left\|(\alpha_1(t)B_0)\cdot n\right\|_{H^{k+1/2}(\partial \Omega)} + \left\|t(\alpha_0(t)\nabla \tilde{u}_V)\cdot n\right\|_{H^{k+1/2}(\partial \Omega)}\right).
        \end{split}
    \end{align}
    Furthermore, $t \mapsto \alpha(t)$ is uniformly bounded in $\mathcal{C}^k$ in any bounded interval containing zero. Since we are only interested in the behavior of $\tilde{\varphi}_t$ for small $t$, we may fix such an interval for the rest of the proof. Therefore, we deduce that the constant appearing in \cref{eq:ineq_regularity} may be chosen to be uniform in $t$. From the continuity of the trace from $H^{k+1}(\Omega)$ to $H^{k+1/2}(\partial \Omega)$, we then obtain 
    $$
    \|\tilde{\varphi}_t\|_{H^{k+2}(\Omega)} \leq C_k \left( \|\alpha_1(t) B_0\|_{H^{k+1}(\Omega)} + \|t\alpha_0(t) \nabla \tilde{u}_V\|_{H^{k+1}(\Omega)}\right)
    $$
    Now, one easily checks by induction, that for $u \in \mathrm{Vec}\left(\bar{\Omega}\right)$ and $A \in \mathcal{C}^\infty\left(\bar{\Omega}, \mathcal{M}_3(\R)\right)$, we have 
    $$
    \|Au\|_{H^k(\Omega)} \leq C_k \|A\|_{W^{k, \infty}(\Omega)}\|u\|_{H^k(\Omega)}.
    $$
    Furthermore, by differentiability of $t \mapsto P_t$ in $\mathcal{C}^\infty$, we have
    \begin{align*}
        &\|\alpha_1(t)\|_{W^{k+1, \infty}(\Omega)} = \|\alpha(t)-\alpha(0)-t\alpha'(0)\|_{W^{k+1, \infty}(\Omega)} = o(t),
        \\
        &\|t\alpha_0(t)\|_{W^{k+1, \infty}(\Omega)} = \|t(\alpha(t)-\alpha(0))\|_{W^{k+1, \infty}(\Omega)} = o(t),
    \end{align*}
    so that 
    $$
    \|\tilde{\varphi}_t\|_{H^{k+2}} = o(t),
    $$
    as claimed.
\end{proof}

Now that we have found the derivative of $t \mapsto (P_t)_1^* B_t$, we need to relate it to the derivative of $t \mapsto (P_t^{-1})_*B_t$. This is achieved using the following lemma.

\begin{lemma}\label{lemma:change_of_pullback}
    Let $X$ be a vector field in $\mathrm{Vec}\left(\bar{\Omega}\right)$. We have 
    $$
    \frac{d}{dt}_{\big{|}t=0} (P_t ^{-1})_* \left((P_t^{-1})^*_1 X\right) = [V, X] - \nabla(V \cdot X) - \curl X \times V.
    $$
\end{lemma}

\begin{proof}
    First, we extend $X$ to a smooth vector field of $\R^3$, which we also denote $X$. Now, we compute 
    $$
    (P_t^{-1})_1^* X = D(P_t^{-1})^T X \circ P_t^{-1}.
    $$
    We have 
    \begin{align*}
        X \circ P_t^{-1} &= X - t(DX) V + o(t),
        \\
        D(P_t^{-1})^T &= I - tDV^T + o(t),
    \end{align*}
    so that 
    $$
    (P_t^{-1})^*_1 X = X - t\left(\left(DV^T\right) X + (DX) V\right) { + o(t)}.
    $$
    Now, combining the identities 
    \begin{align*}
        \nabla(V \cdot X) &= \left(DV^T\right) X + \left(DX^T\right) V,
        \\
        \curl X \times V &= \left(DX - DX^T\right)V,
    \end{align*}
    we obtain 
    $$
    (P_t^{-1})^*_1 X = X - t(\nabla(V \cdot X) + \curl X \times V) + o(t).
    $$
    Finally, defining $Y_t = (P_t^{-1})^*_1 X$ and using
    $$
    (P_t^{-1})_* Y_t = Y_t+ t[V, Y_t] + o(t),
    $$
    we obtain the desired formula.
\end{proof}

We have now all the ingredients to prove the main result of this section.
\begin{proof}[Proof of \cref{th:differentiability_B}]
    We have 
    \begin{align*}
        (P_t^{-1})_* B_t = (P_t^{-1})_* \left((P_t^{-1})^*_1(P_t)^*_1 B_t\right).
    \end{align*}
    Using \cref{cor:diff_phi_tV_B_t,lemma:change_of_pullback}, we thus know that $\R \ni t \mapsto (P_t^{-1})_* B_t \in \mathrm{Vec}\left(\bar{\Omega}\right)$ is differentiable at zero by composition of differentiable maps. Furthermore, its derivative is given by 
    \begin{align*}
        \frac{d}{dt}_{\big{|}t=0} (P_t^{-1})_* B_t &= \frac{d}{dt}_{\big{|}t=0} (P_t)_1^* B_t +[V, B_0] - \nabla(V \cdot B_0) - \curl B_0 \times V
        \\
        &= \nabla u_V + \nabla(V \cdot B_0) + [V, B_0] - \nabla(V \cdot B_0)
        \\
        &= [V, B_0] + \nabla u_V,
    \end{align*}
    as claimed.
\end{proof}

\begin{remark}
    Although it will be simpler to work with paths of diffeomoprhisms to obtain shape differentiability of the Poincaré map, we note that all the techniques used in this section for the shape differentiability of harmonic fields work in a Fréchet differentiability context. That is, we could obtain estimates of the form 
    $$
    \left(I + V\right)^{-1}_* B((I + V)\Omega) = B(\Omega) + [V, B_0] + \nabla u_V + o\left(\|V\|_{C^k}\right),
    $$
    for all $k$ in $\N$, $V$ being a smooth vector field of $\R^3$ in the unit ball of $\mathcal{C}^1(\R^3; \R^3)$.
\end{remark}

\subsection{Shape differentiation of the Poincaré map}\label{sub_sec:shape_diff_poinc}
Now that we have obtained the shape differentiability of the harmonic field, we proceed to compute the shape derivative of its Poincaré map. We denote with an additional $t$ subscript all the objects defined in \cref{sec:definitions} associated with the embedding $\mathcal{E}_t = P_t \circ \mathcal{E}$.
\begin{proposition}\label{prop:general_X'}
    For $|t|<\varepsilon$ sufficiently small, $\mathcal{E}_t$ is admissible and the mapping
    $$
    \begin{cases}
        (-\varepsilon, \varepsilon) \rightarrow \mathrm{Vec}(\partial \Omega) \\
        t \mapsto \left(P_t^{-1}\right)_* X_t
    \end{cases}
    $$
    is differentiable at zero{, where $X_t$ is the harmonic field normalized in the toroidal direction as defined by \cref{eq:def_X}}. Furthermore, if $(\partial_\phi, \partial_\theta)$ is positively oriented, the $\theta$ component of its derivative is given by
    \begin{equation}\label{eq:X'}
        \left(X'_V\right)^\theta = \frac{1}{\sqrt{g}\left(B_{0}^{\phi}\right)^2}B_V' \cdot B_0^\perp, 
    \end{equation}
    and we obtain the same formula with opposite sign if the orientation of the coordinates is reversed.
\end{proposition}

Before proving \cref{prop:general_X'}, we introduce some geometrical notations for vector fields on the boundary. We denote by $\nabla_\Gamma \phi$ (resp. $\nabla_\Gamma \theta$) the vector field of $\partial \Omega$ dual to $d\phi$ (resp. $d\theta$). These are therefore defined by the relations
$$
\nabla_\Gamma \phi \cdot v = d\phi(v), \hspace{1ex} \nabla_\Gamma \theta \cdot v = d\theta(v),
$$
for all vectors $v$ which are tangent to $\partial \Omega$. Contrary to what the notations may suggest, these vector fields are not gradient vector fields, but are only $\curl_\Gamma$-free. This is similar to the fact that $d\phi$ and $d\theta$ are not differentials of global functions, but are closed one-forms. In coordinates we have 
\begin{equation}\label{eq:gradient_coordinates}
    \nabla_\Gamma \phi = g^{\phi \phi}\partial_\phi + g^{\phi \theta}\partial_\theta, \hspace{1ex} \nabla_\Gamma \theta = g^{\theta \phi} \partial_\phi + g^{\theta \theta} \partial_\theta.
\end{equation}
By definition, it is clear that $\nabla_\Gamma \phi$ is orthogonal to $\partial_\theta$, and that $\nabla_\Gamma \theta$ is orthogonal to $\partial_\phi$. Furthermore, if $(\partial_\phi, \partial_\theta)$ is a positively oriented frame on $\partial \Omega$, a straightforward computation in coordinates shows that 
\begin{equation}\label{eq:orthogonality_relations}
    \partial_\phi^\perp = \sqrt{g} \nabla_\Gamma \theta, \hspace{1ex} \partial_\theta^\perp = -\sqrt{g} \nabla_\Gamma \phi,
\end{equation}
with opposite signs if the orientation of the coordinates is reversed. Because of this dependence on orientation for the sign of orthogonal vectors in coordinates, we will often only treat the case where $(\partial_\phi, \partial_\theta)$ is positively oriented. Treating the other case is however a straightforward process, so we will often omit this technicality when writing the main results.

\begin{proof}[Proof of \cref{prop:general_X'}]
    For the first point of the proposition, we observe that $\mathcal{E}_t = P_t \circ \mathcal{E}$ automatically verifies the first two assumptions of \cref{def:admissible_embeddings}. We therefore only need to prove that $B_t^{\phi_t}$ is positive on $\partial \Omega$ for $t$ small enough. To do so, we note that since $\mathcal{E}_t = P_t \circ \mathcal{E}$, we have $\phi_t = \phi \circ P_t^{-1}$, so that 
    \begin{equation}\label{eq:pushforward_component}
        B_t^{\phi_t} = \left(\left(P_t^{-1}\right)_* B_t\right)^\phi { \circ P_t^{-1}}.
    \end{equation}
    We then deduce from the differentiability of $t \mapsto \left(P_t^{-1}\right)_* B_t$ in $\mathrm{Vec}\left(\bar{\Omega}\right)$ and the admissibility of $\mathcal{E}$ that $B_t^{\phi_t}$ is positive for small enough $t$, so that $\mathcal{E}_t$ is admissible.

    Now, using \cref{eq:pushforward_component}, we obtain 
    \begin{align*}
        \left(P_t^{-1}\right)_{*}X_t &= \left(P_t^{-1}\right)_{*}\left(\frac{B_t}{B_t^{\phi_t}}\right)
        \\
        &= \frac{\left(P_t^{-1}\right)_*B_t}{\left(\left(P_t^{-1}\right)_* B_t\right)^\phi}.
    \end{align*}
    Therefore, from the differentiability of $t \mapsto \left(P_t^{-1}\right)_*B_t$ given by \cref{th:differentiability_B} and the fact that $\left(\left(P_t^{-1}\right)_* B_t\right)^\phi$ is nowhere zero for small enough $t$, we obtain that $t \mapsto \left(P_t^{-1}\right) X_t$ is differentiable in $\mathrm{Vec}(\partial \Omega)$ at $t=0$, and its derivative at zero is given by
    $$
    X'_V = \frac{1}{\left(B_0^\phi\right)^2}\left(B_0^\phi B'_V - \left(B'_V\right)^\phi B_0\right).
    $$
    Now, using the fact that $(\partial_\phi, \partial_\theta)$ is positively oriented and \cref{eq:orthogonality_relations}, we obtain
    $$
    B_0^\perp = \sqrt{g} \left(B_0^\phi \nabla_\Gamma \theta - B_0^\theta \nabla_\Gamma \phi\right),
    $$
    so that
    \begin{align*}
        \left(X'_V\right)^\theta &= \frac{1}{\left(B_0^\phi\right)^2}\left(B_0^\phi \left(B'_V\right)^\theta - \left(B'_V\right)^\phi B_0^\theta \right)
        \\
        &= \frac{1}{\left(B_0^\phi\right)^2} \left(B_0^\phi \nabla_\Gamma \theta - B_0^\theta \nabla_\Gamma \phi\right) \cdot \left(\left(B'_V\right)^\phi \partial_\phi + \left(B'_V\right)^\theta \partial_\theta \right)
        \\
        &= \frac{1}{\sqrt{g}\left(B_0^\phi\right)^2} B_0^\perp \cdot B'_V.
    \end{align*}
\end{proof}

We are now able to prove that the Poincaré map is shape differentiable.
\begin{proposition}\label{prop:general_F'}
    The mapping 
    $$
    \begin{cases}
        (-\varepsilon, \varepsilon) \rightarrow \mathrm{Diff}(S^1)
        \\
        t \mapsto \Pi_t
    \end{cases}
    $$
    is differentiable at zero, and its derivative is given by 
    $$
    \Pi'(\mathcal{E};V)(\theta) = \int_{0}^{1} \mathcal{T}(\phi, \theta)(X'_V)^\theta(\phi, \Pi^\phi(\theta))d\phi,
    $$
    where 
    $$
    \mathcal{T}(\phi, \theta) = \exp\left(\int_\phi^1 \partial_\theta X_0^\theta\left(\phi', \Pi^{\phi'}(\theta)\right) d\phi'\right).
    $$
\end{proposition}

\begin{proof}
    Let $x_t(\cdot)$ be the solution to 
    \begin{equation}\label{eq:theta_t}
        \frac{d}{d\phi}x_t(\phi) = X_t^{\theta_t}(\phi, x_t(\phi)),
    \end{equation}
    with $x_t(0) = \theta \in S^1$, so that $\Pi_t^\phi(\theta) = x_t(\phi)$. Using $\theta_t = \theta \circ P_t^{-1}$, we get 
    $$
    X_t^{\theta_t} = \left(\left(P_t^{-1}\right)_*X_t\right)^\theta,
    $$
    so that from \cref{prop:general_X'}, $t \mapsto X_t^{\theta_t}$ is differentiable in $\mathcal{C}^\infty\left(\T^2\right)$. As a consequence, $t \mapsto x_t(\cdot)$ is also differentiable, and we can write 
    \begin{align*}
        x_t(\phi) &= x^{(0)}(\phi) + t x^{(1)}(\phi) + o(t),
        \\
        X_t^{\theta_t}(\phi, \theta) &= X_0^\theta(\phi, \theta) + t \left(X'_V\right)^\theta(\phi, \theta) + o(t),
    \end{align*}
    Where $o(t)$ is here a shorthand for a function whose $\mathcal{C}^k$ norms on all compact subsets are $o(t)$. Since $\Pi_t(\theta) = x_t(1)$, we obtain that $t \mapsto \Pi_t$ is differentiable, and its derivative is given by $\Pi'(\mathcal{E};V) = x^{(1)}(1)$. Injecting the expansions for $X_t^{\theta_t}$ and ${x_t}$ in \cref{eq:theta_t}, we obtain 
    $$
    \frac{d}{d\phi}x^{(0)}(\phi) + t\frac{d}{d\phi}x^{(1)}(\phi) = X_0^\theta\left(\phi, x^{(0)}(\phi)\right) + t\left[\partial_\theta X_0^\theta\left(\phi, x^{(0)}(\phi)\right)x^{(1)}(\phi) + \left(X'_V\right)^\theta\left(\phi, x^{(0)}(\phi)\right)\right] + o(t),
    $$
    so that $x^{(1)}(\cdot)$ solves the following linear equation with a drift term
    $$
    \frac{d}{d\phi}x^{(1)}(\phi) = \partial_\theta X_0^\theta\left(\phi, \Pi_0^\phi(\theta)\right)x^{(1)}(\phi) + \left(X'_V\right)^\theta\left(\phi, \Pi_0^\phi(\theta)\right),
    $$
    with $x^{(1)}(0) = 0$. Using Duhamel's formula, we thus obtain 
    $$
    x^{(1)}(1) = \int_{0}^{1} e^{\int_{\phi}^{1}\partial_\theta X_0^\theta\left(\phi', \Pi_0^{\phi'}(\theta)\right)d\phi'} \left(X'_V\right)^\theta(\phi, \Pi_0^\phi(\theta)) d\phi,
    $$
    which is the desired result.
\end{proof}

In the case where $B_0$ is linearized in the $(\phi, \theta)$ coordinates, that is, $(B_0)_{|\partial \Omega} = \chi(\partial_\phi + \omega \partial_\theta)$ where $\chi$ is a smooth function of $\partial\Omega$ and $\omega$ is in $\R$, we have the following formulas for $(X'_V)^\theta$ and $\Pi'(\mathcal{E};V)$.

\begin{proposition}\label{prop:derivatives_linearizable}
    Suppose there exist $\chi$ in $\mathcal{C}^\infty(\partial \Omega)$ positive and a real number $\omega$ such that $(B_0)_{|\partial \Omega}= \chi(\partial_\phi + \omega \partial_\theta)$ with $(\partial_\phi, \partial_\theta)$ positively oriented. Let $\tilde n$ be a smooth extension of $n$ to $\R^3$ and $\vec{\omega}=(1, \omega)^T$. Decomposing $V \in \mathrm{Vec}(\R^3)$ as $V=f \tilde{n} + V_\Gamma$ where $(V_\Gamma)_{|\partial \Omega}$ is tangent to $\partial \Omega$, we have 
    \begin{equation}\label{eq:X'_linearized}
        (X'_V)^\theta = f\frac{2\mathbb{II}(B_0,B_0^\perp)}{\sqrt{g}\chi^2} + \left\langle \vec{\omega}, \nabla_{\T^2}\left(\omega V_\Gamma^\phi - V_\Gamma^\theta\right) \right\rangle + \frac{1}{\sqrt{g}\chi^2}B_0^\perp \cdot \nabla_\Gamma u_V,
    \end{equation}
    where $\mathbb{II}$ is the second fundamental form of $\partial \Omega$. Furthermore, we also have 
    $$
    \Pi'(\mathcal{E}; V)(\theta) = \int_{0}^{1}(X'_V)^\theta(\phi, \theta+\omega \phi)d\phi.
    $$
\end{proposition}

\begin{proof}
    From \cref{th:differentiability_B} and \cref{prop:general_X'}, we have 
    \begin{align}
        \left(X'_V\right)^\theta &= \frac{1}{\sqrt{g}\chi^2} B_0^\perp \cdot B_V' \nonumber
        \\
        &= \frac{1}{\sqrt{g}\chi^2} B_0^\perp \cdot \left([V, B_0] + \nabla u_V\right) \nonumber
        \\
        &= \frac{1}{\sqrt{g} \chi^2} B_0^\perp \cdot [f \tilde{n}, B_0] + \frac{1}{\sqrt{g} \chi^2} B_0^\perp \cdot [V_\Gamma, B_0] + \frac{1}{\sqrt{g}\chi^2} B_0^\perp \cdot \nabla_\Gamma u_V. \label{eq:decomp_X'}
    \end{align}
    For the first term of \cref{eq:decomp_X'}, we have 
    $$
    [f \tilde{n}, B_0] = f[\tilde{n}, B_0] - \left(B_0 \cdot \nabla f\right) \tilde{n},
    $$
    so that 
    $$
    B_0^\perp \cdot [f \tilde{n}, B_0] = f B_0^\perp \cdot [\tilde{n}, B_0].
    $$
    We note that since $B_0$ is in $\mathrm{Vec}\left(\bar{\Omega}\right)$ and $\Omega$ is a smooth domain, we may extend $B_0$ and $B_0^\perp$ to smooth vector fields of $\R^3$ when necessary. Now, denoting by $\nabla_X Y$ the covariant derivative of a vector field $Y$ in the direction $X$ in $\R^3$, and using the fact that the Levi--Civita connection is torsion free, we have
    $$
    [\tilde{n}, B_0] = \nabla_{\tilde{n}} B_0 - \nabla_{B_0} \tilde{n}.
    $$
    We also have 
    $$
    \nabla (B_0 \cdot \tilde{n}) = \nabla_{\tilde{n}} B_0 + \nabla_{B_0} \tilde{n} + \tilde{n} \times \curl B_0 + B_0 \times \curl \tilde{n}.
    $$
    It is straightforward to see that the tangential part of $[\tilde{n}, B_0]$ does not depend on the choice of extension of the normal, so that we may choose $\tilde{n}$ in a specific way. Since $\partial \Omega$ is smooth, we know that the signed distance to $\partial \Omega$ (which we denote $\sigma_{\partial \Omega}$) is smooth in a neighborhood $U$ of $\partial \Omega$. Let $K$ be a compact subset of $U$ containing a neighborhood of $\partial \Omega$, and $\eta$ be a smooth positive function which is equal to one in $K$, and has support included in $U$. Then, $\tilde{n} := \eta \nabla \sigma_{\partial \Omega}$ is smooth extension of the normal, and $\curl \tilde{n} = 0$ in $K$. Therefore, we have $\curl{\tilde{n}}=0$ on $\partial \Omega$. Furthermore, we also have $\curl B_0 = 0$ in $\bar{\Omega}$. As such, using that $B_0^\perp$ is tangent to $\partial \Omega$, and that $B_0 \cdot \tilde{n}$ is equal to zero on $\partial \Omega$, we have 
    $$
    B_0^\perp \cdot \nabla_{\tilde{n}} B_0 + B_0^\perp \cdot \nabla_{B_0}\tilde{n} = B_0^\perp \cdot \nabla(B_0 \cdot \tilde{n}) = 0,
    $$
    so that 
    $$
    B_0^\perp \cdot [\tilde{n}, B_0] = -2 B_0^\perp \cdot \nabla_{B_0} \tilde{n}.
    $$
    Now, using the fact that the Levi--Civita is compatible with the metric, we write 
    $$
    B_0 \cdot \nabla(B_0^\perp \cdot \tilde{n}) = \nabla_{B_0} B_0^\perp \cdot \tilde{n} + B_0^\perp \cdot \nabla_{B_0} \tilde{n}.
    $$
    Therefore, since $B_0 \cdot \nabla(B_0^\perp \cdot \tilde{n})$ vanishes on $\partial \Omega$, we have 
    \begin{align*}
        B_0^\perp \cdot \nabla_{B_0}\tilde{n} &= -\nabla_{B_0}B_0^\perp \cdot \tilde{n}
        \\
        &= -\mathbb{II}(B_0, B_0^\perp).
    \end{align*}
    We refer to \cite{lee_introduction_2018}[Section 8] for the definition of the second fundamental form of 1-codimensional manifolds using the Levi--Civita connection. As a consequence, the first term of \cref{eq:decomp_X'} is given by 
    $$
    \frac{1}{\sqrt{g}\chi^2} [f \tilde{n}, B_0] = f \frac{2 \mathbb{II}(B_0, B_0^\perp)}{\sqrt{g}\chi^2}.
    $$
    Now, we compute the second term of \cref{eq:decomp_X'}. Since $V_\Gamma$ and $B_0$ are tangent vector fields, the tangential part of $[V_\Gamma, B_0]$ is given by the Lie bracket of $V_\Gamma$ and $B_0$ as vector fields of $\partial \Omega$, which we denote $[V_\Gamma, B_0]_{\partial \Omega}$. We have $B_0 = \chi(\partial_\phi + \omega \partial_\theta) = \chi X_0$, so that 
    $$
    [V_\Gamma, B_0]_{\partial \Omega} = V_\Gamma \cdot \left(\nabla_\Gamma \chi\right) X_0 + \chi [V_\Gamma, X_0]_{\partial \Omega}.
    $$
    Since $X_0$ is collinear to $B_0$ it is orthogonal to $B_0^\perp$ which implies
    $$
    B_0^\perp \cdot [V_\Gamma, B_0] = \chi B_0^\perp \cdot [V_\Gamma, X_0]_{\partial \Omega}.
    $$
    Now, we write in coordinates 
    \begin{align*}
        X_0 &= \partial_\phi + \omega \partial_\theta,
        \\
        V_\Gamma &= V_\Gamma^\phi \partial_\phi + V_\Gamma^\theta \partial_\theta,
    \end{align*}
    which gives us 
    $$
    [V_\Gamma, X_0]_{\partial \Omega} = -(\partial_\phi V_\Gamma^\phi + \omega \partial_\theta V_\Gamma^\phi) \partial_\phi - (\partial_\phi V_\Gamma^\theta + \omega \partial_\theta V_\Gamma^\theta)\partial_\theta.
    $$
    Finally, using $B_0^\perp = \sqrt{g}\chi(\nabla_\Gamma \theta - \omega \nabla_\Gamma \phi)$, we get 
    \begin{align*}
        B_0^\perp \cdot [V_\Gamma, B_0] &= -\sqrt{g}\chi^2 (\nabla_\Gamma \theta - \omega \nabla_\Gamma \phi) \cdot \left[(\partial_\phi V_\Gamma^\phi + \omega \partial_\theta V_\Gamma^\phi) \partial_\phi + (\partial_\phi V_\Gamma^\theta + \omega \partial_\theta V_\Gamma^\theta)\partial_\theta\right]
        \\
        &= -\sqrt{g}\chi^2 \left(\partial_\phi V_\Gamma^\theta + \omega \partial_\theta V_\Gamma^\theta - \omega \partial_\phi V_\Gamma^\phi - \omega^2 \partial_\theta V_\Gamma^\phi\right),
        \\
        &= \sqrt{g}\chi^2 \left\langle{\vec{\omega}, \nabla_{\T^2}\left(\omega V_\Gamma^\phi - V_\Gamma^\theta\right)} \right\rangle,
    \end{align*}
    which completes the proof of the first statement. The second result is then a simple consequence of \cref{prop:general_F'} and the fact that, since $(X_0)_{|\partial \Omega} = \partial_\phi + \omega \partial_\theta$, we have $\Pi^\phi(\theta) = \theta + \omega \phi$.
\end{proof}

\section{The axisymmetric case}\label{sec:axisymmetric}
In this section, we consider the embedding of the standard axisymmetric torus defined in Cartesian coordinates by
\begin{equation}\label{eq:axisymmetric_embedding}
    \mathcal{E}(\phi, \theta) = \left((R_T + r_P \cos(2\pi\theta))\cos(2\pi\phi), (R_T + r_P \cos(2\pi\theta))\sin(2\pi\phi), r_P\sin(2\pi\theta)\right),
\end{equation}
where $R_T$ and $r_P$ are the major and minor radius respectively with $r_P < R_T$. We also denote 
$$
R(\theta) = R_T + r_P\cos(2\pi\theta),
$$
which is the distance of the point $\mathcal{E}(\phi, \theta)$ to the $z$-axis. The aim of this section is to prove the following theorem.

\begin{theorem}\label{th:zero_derivative_axisymmetric}
    Let $\mathcal{E}$ be as described above. We have for all $V$ in $\mathrm{Vec}(\R^3)$ 
    $$
    \Pi'(\mathcal{E}; V) = 0.
    $$
\end{theorem}

{
\begin{remark}\label{rem:th_axisymmetric}
    Although we only work with the standard axisymmetric embedding given by \cref{eq:axisymmetric_embedding} for simplicity, one can easily adapt all the proofs so that \cref{th:zero_derivative_axisymmetric} is in fact valid for all axisymmetric embeddings.
\end{remark}
}

We begin by computing the relevant geometric objects associated with this embedding. The basis vectors of the coordinates $(\phi, \theta)$ are given by
\begin{align*}
    \partial_\phi &= -2\pi R(\theta)\sin(2\pi\phi)\partial_x + 2\pi R(\theta)\cos(2\pi\phi)\partial_y, 
    \\
    \partial_\theta &= -2\pi r_P\sin(2\pi\theta)\cos(2\pi\phi)\partial_x - 2\pi r_P\sin(2\pi\theta)\sin(2\pi\phi)\partial_y + 2\pi r_P\cos(2\pi\theta)\partial_z,
\end{align*}
so that 
$$
g = 4\pi^2 R(\theta)^2 d\phi^2 + 4\pi^2 r_P^2 d\theta^2.
$$
We also deduce $\sqrt{g} = 4\pi^2 r_P R(\theta)$. One also verifies that
$$
n = \cos(2\pi\theta)\cos(2\pi\phi)\partial_x + \cos(2\pi\theta)\sin(2\pi\phi)\partial_y + \sin(2\pi\theta)\partial_z.
$$
Computing the second-order derivatives of $\mathcal{E}$, we get 
\begin{align*}
    \partial_\phi^2 \mathcal{E} &= -4\pi^2 R(\theta) \cos(2\pi\phi)\partial_x -4\pi^2 R(\theta) \sin(2\pi\phi)\partial_y,
    \\
    \partial_\phi \partial_\theta \mathcal{E} &= 4\pi^2 r_P \sin(2\pi\theta)\sin(2\pi\phi)\partial_x - 4\pi^2 r_P \sin(2\pi\theta)\cos(2\pi\phi)\partial_y,
    \\
    \partial_\theta^2 \mathcal{E} &= -4\pi^2 r_P \cos(2\pi\theta)\cos(2\pi\phi)\partial_x - 4\pi^2r_P \cos(2\pi \theta)\sin(2\pi\phi)\partial_y - 4\pi^2r_P \sin(2\pi\theta)\partial_z,
\end{align*}
so that
\begin{align}
    \begin{split}\label{eq:II_axisymmetric}
        \mathbb{II} &= \left(\partial_\phi^2 \mathcal{E} \cdot n\right) d\phi^2 + 2\left(\partial_\phi \partial_\theta \mathcal{E} \cdot n\right) d\phi d\theta + \left(\partial_\theta^2 \mathcal{E} \cdot n\right) d\theta^2,
        \\
        &= -4\pi^2 R(\theta) \cos(2\pi\theta) d\phi^2 -4\pi^2 r_P d\theta^2.
    \end{split}
\end{align}

We now turn to the underlying domain $\Omega$, and the associated harmonic field. $\mathcal{E}(\T^2)$ bounds the domain 
$$
\Omega=\{ \left( (R_T+r_P x)\cos (2\pi\phi), (R_T+r_P x)\sin (2\pi\phi), r_P y \right) \in \R^3 \mid (\phi, x, y) \in S^1 \times D^2 \}.
$$
In this case, the harmonic field of $\Omega$ is explicitly known, and is given by the formula 
$$
B(\Omega) = \frac{1}{2\pi}\left(-\frac{y}{x^2+y^2}\partial_x + \frac{x}{x^2+y^2}\partial_y\right),
$$
where the $1/2\pi$ constant ensures that $B(\Omega)$ has unit circulation along positively oriented toroidal loops. Moreover, the restriction of $B(\Omega)$ to the boundary is given by 
\begin{equation}\label{eq:B_axisymmetric}
    B(\Omega)_{|\partial \Omega} = \frac{1}{4\pi^2 R(\theta)^2}\partial_\phi.
\end{equation}
It is then clear that $\mathcal{E}$ is indeed an admissible embedding. We are now able to prove the following {proposition}.
\begin{proposition}\label{prop:axisymmetric_objects}
    Let $\mathcal{E}$, $B(\Omega)$ and $(\phi, \theta)$ be as defined above. We then have 
    \begin{itemize}
        \item $X(\mathcal{E}) = \partial_\phi$,
        \item $\Pi(\mathcal{E}) = id$,
        \item $\mathbb{II}(B(\Omega), B(\Omega)^\perp) = 0$.
    \end{itemize}
\end{proposition}

\begin{proof}
    The first two statements are straightforward using \cref{eq:B_axisymmetric}. As for the third statement, using the fact that the coordinates $(\phi, \theta) $ are orthogonal, we know that $B(\Omega)^\perp$ is colinear to $\partial_\theta$. Furthermore, we know from \cref{eq:II_axisymmetric} that the second fundamental form is diagonalized in the coordinates $(\phi, \theta)$, which gives us the desired result.
\end{proof}

\begin{corollary}\label{cor:expression_derivative_axisymmetric}
    Let $\mathcal{E}$ and $\Omega$ be as defined above, and $u_V$ be the solution to \cref{eq:u}. Then 
    $$
    \Pi'(\mathcal{E}; V)(\theta) = \frac{R(\theta)^2}{r_P^2}\int_{0}^{1} \partial_\theta u_V(\phi, \theta) d\theta.
    $$
\end{corollary}

\begin{proof}
    Using \cref{prop:axisymmetric_objects}, we know that we can apply \cref{prop:derivatives_linearizable} with $\omega = 0$. We decompose $V \in \mathrm{Vec}\left(\R^3\right)$ as $V = f \tilde{n} + V_\Gamma$, where $\tilde{n}$ is a smooth extension of $n$ to $\R^3$, and $V_\Gamma$ is tangent to $\partial \Omega$. Since $(\partial_\phi, \partial_\theta)$ is positively oriented, we know from \cref{prop:derivatives_linearizable} and \cref{prop:axisymmetric_objects} that
    $$
    (X'_V)^\theta = -\partial_\phi V_\Gamma^\theta + \frac{1}{\sqrt{g}(B(\Omega)^\phi)^2} B(\Omega)^\perp \cdot \nabla_\Gamma u_V.
    $$
    Using the fact that $v \mapsto v^\perp$ is an isometry on each tangent plane of $\partial \Omega$ {and that $\left(u^\perp\right)^\perp = -u$}, we find  
    \begin{align*}
        B(\Omega)^\perp \cdot \nabla_\Gamma u_V &= {-}B(\Omega) \cdot \left(\nabla_\Gamma u_V\right)^\perp 
        \\
        &= {-}\left(B(\Omega)^\phi \partial_\phi\right) \cdot \left(\frac{1}{\sqrt{g}}\partial_\phi u_V \partial_\theta - \frac{1}{\sqrt{g}}\partial_\theta u_V \partial_\phi\right)
        \\
        &= B(\Omega)^\phi \partial_\theta u_V \frac{g_{\phi \phi}}{\sqrt{g}},
    \end{align*}
    where we used $\nabla_\Gamma u_V = \partial_\phi u_V \nabla_\Gamma \phi + \partial_\theta u_V \nabla_\Gamma \theta$, as well as $\nabla_\Gamma \phi^\perp = 1/\sqrt{g}\partial_\theta$ and $\nabla_\Gamma \theta^\perp = -1/\sqrt{g}\partial_\phi$, which are simple consequences of \cref{eq:gradient_coordinates,eq:orthogonality_relations}. Therefore, we obtain 
    \begin{align*}
        \left(X'_V\right)^\theta &= -\partial_\phi V_\Gamma^\theta {+} \frac{g_{\phi \phi}}{(\det g) B(\Omega)^\phi} \partial_\theta u_V
        \\
        &= -\partial_\phi V_\Gamma^\theta {+} \frac{R(\theta)^2}{r_P^2}\partial_\theta u_V.
    \end{align*}
    Using once again \cref{prop:derivatives_linearizable}, we obtain 
    \begin{align*}
        \Pi'(\mathcal{E}; V) &= \int_0^1 \left(X'_V\right)^\theta(\phi, \theta) d\phi
        \\
        &= \frac{R(\theta)^2}{r_P^2}\int_{0}^{1} \partial_\theta u_V(\phi, \theta) d\phi.
    \end{align*}
\end{proof}

The proof of \cref{th:zero_derivative_axisymmetric} follows from \cref{cor:expression_derivative_axisymmetric} taking into account suitable symmetry properties described in the following lemmas.

\begin{lemma}\label{lemma:zero_average_condition}
Let $\Omega$ be as defined above and $u_V$ be as in \cref{th:differentiability_B}. Then, for all $\theta$ in $S^1$, we have
\begin{equation}\label{eq:zero average condition}
\int_{0}^{1} \nabla u_V(\phi, \theta) \cdot n d\phi = 0.
\end{equation}
\end{lemma}

\begin{proof}
    We recall that $u_V$ is a harmonic function of $\Omega$ satisfying the boundary condition
    $$
    \nabla u_V \cdot n = \div_{\Gamma}(B(\Omega) (V \cdot n)).
    $$
    A quick computation in coordinates shows that $\div_{\Gamma}(B(\Omega))=0$, so that 
    $$
    \div_{\Gamma}(B(\Omega) (V \cdot n)) = B(\Omega) \cdot \nabla_{\Gamma}(V \cdot n).
    $$
    {Furthermore, since $B(\Omega) = 1/(4\pi^2 R(\theta)^2)\partial_\phi$, we have $B(\Omega) \cdot \nabla_{\Gamma}(V \cdot n) = 1/(4\pi^2 R(\theta)^2)\partial_\phi(V \cdot n)$. Finally, we deduce 
    \begin{align*}
        \int_0^1 \nabla u_V (\phi, \theta) \cdot n d\phi  &= \frac{1}{4\pi^2R(\theta)^2} \int_0^1 \partial_\phi (V \cdot n)(\phi, \theta) d\phi
        \\
        &= 0.
    \end{align*}}
\end{proof}

Since $u_V$ is harmonic in $\Omega$, we know that $\nabla u_V \cdot n$ must be of zero average on $\partial \Omega$. \cref{lemma:zero_average_condition} then tells us that $\nabla u_V \cdot n$ must moreover be of zero average along any toroidal loop. In particular, $u_V$ may not be any harmonic function of $\bar{\Omega}$. This fact is then used to prove the following lemma, for which we introduce the notation
$$
\fint_{\partial \Omega} f = \frac{\int_{\partial \Omega} f}{|\partial \Omega|} =  \frac{\int_{\partial \Omega} f }{\int_{\partial \Omega} 1},
$$
where $f$ is an integrable function on $\partial \Omega$.
\begin{lemma}\label{lemma:toroidal_averaging}
Let $\Omega$ be as defined above and $u_V$ be as in \cref{th:differentiability_B}. Let $f$ be a smooth function on $\partial \Omega$ such that $\partial_\phi f=0$. We have
\begin{equation}\label{eq:toroidal averaging}
\int_{\partial \Omega} f u_V = \fint_{\partial \Omega} f \int_{\partial \Omega} u_V.
\end{equation}
\end{lemma}

\begin{proof}
First, suppose that $f$ has zero average on $\partial \Omega$. We then define $v$ as the zero average solution to
$$
\begin{cases}
\Delta v = 0 \hspace{0.25cm} \text{in } \Omega,
\\
\nabla v \cdot n = f \hspace{0.25cm} \text{on } \partial \Omega.
\end{cases}
$$
We now show that $\partial_\phi v=0$. Indeed, call $R_\Phi$ the rotation of angle $\Phi$ around the $z$-axis. We then get that $v \circ R_\Phi$ satisfies
$$
\Delta(v \circ R_\Phi) = (\Delta v) \circ R_\Phi = 0,
$$
because $R_\Phi$ is an isometry, and
$$
\nabla (v \circ R_\Phi) \cdot n = (\nabla v \cdot n) \circ R_\Phi = f,
$$
because $R_\Phi$ is an isometry which leaves $\Omega$ unchanged and $\partial_\phi f = 0$. Therefore, $v$ and $v \circ R_\Phi$ satisfy the same PDE, and have the same average. As a consequence, $v \circ R_\Phi = v$ for all $\Phi$, meaning that $\partial_\phi v =0$.

Now, using the equations satisfied by $u_V$ and $v$, and \cref{lemma:zero_average_condition}, we get
\begin{align*}
\int_{\partial \Omega} f u_V &= \int_{\partial \Omega} \left(\nabla v \cdot n\right) u_V
\\
&= \int_{\Omega} \nabla v \cdot \nabla u_V
\\
&= \int_{\partial \Omega} v \left(\nabla u_V \cdot n\right)
\\
&= 4\pi^2 r_P \int_{0}^{1} R(\theta) v(\theta) \int_{0}^{1} \nabla u_V \cdot n (\phi, \theta) d\phi d\theta
\\
&= 0.
\end{align*}
Finally, if we now take any $f$ in $\mathcal{C}^\infty(\partial \Omega)$, we can repeat the procedure with $f - \fint_{\partial \Omega} f$, and get
$$
\int_{\partial \Omega} \left[\left(f - \fint_{\partial \Omega} f\right) u_V\right] = 0,
$$
which gives us our desired result.
\end{proof}

We now have all the ingredients to prove \cref{th:zero_derivative_axisymmetric}.

\begin{proof}[Proof of \cref{th:zero_derivative_axisymmetric}]
To prove that $\Pi'(\mathcal{E};V)$ vanishes, we use \cref{lemma:toroidal_averaging} on approximations of $\delta_{\theta_0}$. We define $\tilde{f}_{\theta_0, \varepsilon} \in \mathcal{C}^\infty(S^1)$ so that $4\pi^2 r_P R \tilde{f}_{\theta_0, \varepsilon}$ is a family of smooth approximations of the Dirac at $\theta_0$. We may take for example
$$
R(\theta)\tilde{f}_{\theta_0, \varepsilon}(\theta) = C_\varepsilon \sum_{k \in \Z} \exp \left(-\frac{\varepsilon^2(\theta-\theta_0-k)^2}{2}\right),
$$
with $C_\varepsilon$ chosen so that
$$
\int_{0}^{1} 4\pi^2 r_P R(\theta) \tilde{f}_{\theta_0, \varepsilon}(\theta) d\theta = 1.
$$
Then, defining $f_{\theta_0, \varepsilon}(\phi, \theta) = \tilde{f}_{\theta_0, \varepsilon}(\theta)$ and using $\sqrt{g} = 4\pi^2r_P R(\theta)$, we get
\begin{align*}
\int_{\partial \Omega} f_{\theta_0, \varepsilon} u_V &= \int_{0}^{1} \int_{0}^{1} 4\pi^2 r_P R(\theta) \tilde{f}_{\theta_0, \varepsilon}(\theta) u_V(\phi, \theta) d\theta d\phi
\\
&\xrightarrow{\varepsilon \rightarrow 0} \int_{0}^{1} u_V(\phi, \theta_0) d\phi.
\end{align*}
On the other hand, \cref{lemma:toroidal_averaging} gives us
\begin{align*}
\int_{\partial \Omega} f_{\theta_0, \varepsilon} u_V &= \fint_{\partial \Omega} f_{\theta_0, \varepsilon} \int_{\partial \Omega} u_V
\\
&= \fint_{\partial \Omega} u_V.
\end{align*}
As a consequence, we have for all $\theta$ in $S^1$
$$
\int_{0}^{1} u_V(\phi, \theta) d\phi = \fint_{\partial \Omega} u_V,
$$
and thus
$$
\int_{0}^{1}\partial_\theta u_V(\phi, \theta)d\phi = 0.
$$
We then conclude using the formula of $\Pi'(\mathcal{E}; V)$ given in \cref{cor:expression_derivative_axisymmetric}.
\end{proof}

\section{The diophantine case}\label{sec:diophantine}

In this section, we suppose that $\mathcal{E}$ is an admissible embedding such that, in the corresponding coordinates $(\phi, \theta)$, we have $B_{|\partial \Omega} = \chi\left(\partial_\phi + \omega \partial_\theta\right)$, where $\chi$ is a smooth function on the boundary and $\omega$ is a diophantine number, that is, there exist $C, \tau$ positive constants such that, for all $p/q \in \Q$
\begin{equation}\label{eq:diophantine}
    \left| \omega - p/q \right| \geq C |q|^{-(\tau + 1)}.
\end{equation}
We note that this definition of diophantine numbers implies two other inequalities which will be used in this section. The first one, which is generally used for cohomological equations in the continuous context, is the following. For all $n \neq 0$ in $\Z^2$, we have 
\begin{equation}\label{eq:diophantine_continuous}
    \left|\vec{\omega} \cdot n\right| \geq C |n|^{-\tau},
\end{equation}
with $\vec{\omega} = (1, \omega)^T$. The second one, which is generally more common in discrete contexts, is the following. For all $q \neq 0$ in $\Z$, we have 
\begin{equation}\label{eq:diophantine_discrete}
    \left|e^{2\pi i \omega q} - 1\right| \geq C|q|^{-\tau},
\end{equation}
where $C$ is not necessarily the same constant as before. To obtain this inequality, we write using \cref{eq:diophantine}
$$
\inf_{p \in \Z}\left|\omega q - p\right| \geq C |q|^{-\tau}.
$$
The quantity on the right-hand side of this inequality is the distance between $\omega q$ and $0$ in $S^1 = \R/\Z$ using the quotient metric induced from the usual metric on $\R$. This metric is then equivalent to the metric on $S^1$ when seen as the unit circle in $\C$, which gives us \cref{eq:diophantine_discrete}.

In this section, we prove the following theorem
\begin{theorem}\label{th:surjectivity_diophantine}
    Suppose $\mathcal{E}$ is an admissible embedding with associated domain $\Omega$ and coordinates $(\phi, \theta)$ verifying the following hypotheses.
    \begin{enumerate}
        \item There exists $\chi$ in $\mathcal{C}^\infty(\partial \Omega)$ and a diophantine number $\omega$ such that $B(\Omega)_{|\partial \Omega} = \chi(\partial_\phi + \omega \partial_\theta)$. \label{hyp:diophantine}
        \item $\mathbb{II}\left(B(\Omega), B(\Omega)^\perp\right)$ vanishes nowhere on $\partial \Omega$. \label{hyp:second_fundamental_form}
    \end{enumerate}
    Then, the mapping $V \mapsto \Pi'(\mathcal{E};V)$ is surjective from $\mathrm{Vec}(\R^3)$ to $\mathcal{C}^\infty(S^1)$.
\end{theorem}

\begin{remark}\label{rem:assumption}
    For a point $x$ of $\partial \Omega$, the second fundamental form $\mathbb{II}_x$ at $x$ is related to the shape operator $S_x$ by $\mathbb{II}_x(u, v) = S_x(u) \cdot v$. Since $\mathbb{II}_x$ is a symmetric bilinear form, $S_x$ is self adjoint and there exist two orthonormal eigenvectors $E_1$ and $E_2$ of $T_x \partial \Omega$ with associated eigenvalues $\kappa_1$ and $\kappa_2$. $E_i$ are the principal directions at $x$, and $\kappa_i$ the related principal curvatures. We can also assume that $(E_1, E_2)$ is positively oriented on $T_x \partial \Omega$. If we decompose $B(\Omega)$ at $x$ as 
    $$
    B(\Omega) = \alpha_1 E_1 + \alpha_2 E_2,
    $$
    then 
    $$
    \mathbb{II}_x(B(\Omega), B(\Omega)^\perp) = \alpha_1 \alpha_2 (\kappa_2 - \kappa_1).
    $$
    Therefore, we find that $\mathbb{II}(B(\Omega), B(\Omega)^\perp)$ does not vanish at $x$ if and only if two conditions are met:
    \begin{itemize}
        \item $x$ is not an umbilical point of $\partial \Omega$, that is $\kappa_1 \neq \kappa_2$.
        \item $B(\Omega)$ is not in a principal direction at $x$.
    \end{itemize}
\end{remark}

The proof of \cref{th:surjectivity_diophantine} comes in two steps. First, we prove that by choosing $V$ tangent to the boundary, we can generate any $\Pi'(\mathcal{E};V)$ which is of zero average on $S^1$. This result is not surprising as having $V$ tangent to the boundary amounts to changing the coordinates on $\partial \Omega$, which in turn change the Poincaré map by a conjugation by a diffeomorhism on $S^1$. As such, tangent deformations do not generate a change in the rotation number, and the only possible changes in $\Pi_t$ have zero average at first-order in $t$. Then, the image of $V \mapsto \Pi(\mathcal{E};V)$ has co-dimension at most one in $\mathcal{C}^\infty(S^1)$, and one only needs to find a normal deformation which generates $\Pi'(\mathcal{E};V)$ with nonzero average. This is achieved using the assumption on $\mathbb{II}\left(B(\Omega), B(\Omega)^\perp\right)$.

\begin{proposition}\label{prop:surjectivity_zero_average}
    For all $\mu$ in $\mathcal{C}^\infty(S^1)$ such that 
    $$
    \int_{S^1}\mu = 0,
    $$
    there exists $V_\Gamma$ in $\mathrm{Vec}(\R^3)$ such that $(V_\Gamma)_{|\partial \Omega}$ is tangent to $\partial \Omega$ and $\Pi'(\mathcal{E};V_\Gamma) = \mu$.
\end{proposition}

\begin{proof}
    Let $\mu$ in $\mathcal{C}^\infty(S^1)$ be given in Fourier basis by 
    $$
    \mu(\theta) = \sum_{n \in \Z} \hat{\mu}_n e^{2\pi i n\theta},
    $$
    with $\hat{\mu}_0 = 0$. Since $\mu$ is real-valued, we have $(\hat{\mu}_{-n})^* = \hat{\mu}_{n}$. We define $\hat{\Phi}_n = (2\pi i n \omega)/(e^{2\pi i n \omega}-1) \hat{\mu}_n$ for $n\neq 0$, and
    $$
    \Phi(\phi, \theta) = \sum_{n \in \Z\backslash\{0\}} \hat{\Phi}_{n} e^{2\pi i n \theta}.
    $$
    Using the discrete diophantine condition on $\omega$ given by \cref{eq:diophantine_discrete}, we get 
    \begin{align*}
        \left| \frac{2\pi i n \omega}{e^{2\pi i n \omega}-1}\hat{\mu}_n \right| &\leq C \left|e^{2\pi i \omega n}-1\right|^{-1} |n| \left|\hat{\mu}_n\right|,
        \\
        &\leq C |n|^{\tau+1} \left|\hat{\mu}_n\right|,
    \end{align*}
    so that $\Phi$ is smooth on $\T^2$. Furthermore, using the symmetries of $\hat{\mu}_n$, it is straightforward that $\Phi$ is also real-valued. Now, define $V_\Gamma$ as a smooth extension of $-\varphi \partial_\theta$ to $\R^3$, where $\varphi$ is the zero average solution to 
    $$
    \langle \vec{\omega}, \nabla_{\T^2} \varphi \rangle = \Phi.
    $$
    This solution is known to exist using the continuous diophantine condition given by \cref{eq:diophantine_continuous} and the fact that $\Phi$ has zero average on $\T^2$. Moreover, we have $V_\Gamma \cdot n = 0$ so that the solution $u_{V_\Gamma}$ to \cref{eq:u} is constant. Using \cref{prop:derivatives_linearizable}, we get 
    \begin{align*}
    (X'_{V_\Gamma})^\theta &= \left\langle \vec{\omega}, \nabla_{\T^2}\left(\omega V_\Gamma^\phi - V_\Gamma^\theta\right) \right\rangle
    \\
    &= \langle \vec{\omega}, \nabla_{\T^2} \varphi \rangle
    \\
    &= \Phi. 
    \end{align*}
    Now, we compute using \cref{prop:derivatives_linearizable}
    \begin{align*}
        \Pi'(\mathcal{E}; V_\Gamma)(\theta) &= \int_{0}^{1} \left(X'_{V_\Gamma}\right)^\theta(\phi, \theta + \omega \phi)d\phi 
        \\
        &= \int_{0}^{1} \Phi(\phi, \theta + \omega \phi)d\phi
        \\
        &= \sum_{n \in \Z\backslash\{0\}} \hat{\Phi}_n \int_{0}^{1} e^{2\pi i n(\theta + \omega \phi)} d\phi 
        \\
        &= \sum_{n \in \Z\backslash\{0\}} \hat{\Phi}_n \frac{e^{2\pi i n \omega}-1}{2\pi i n} e^{2\pi i n \theta}
        \\
        &= \sum_{n \in \Z} \hat{\mu}_n e^{2\pi i {n}\theta}
        \\
        &= \mu(\theta).
    \end{align*}
\end{proof}

\begin{proof}[Proof of \cref{th:surjectivity_diophantine}]
    From \cref{prop:surjectivity_zero_average}, we know that the image of $V \mapsto \Pi'(\mathcal{E};V)$ contains all the smooth zero average functions on $S^1$. By linearity, we therefore only need to find one deformation which produces a derivative of the Poincaré map with nonzero average. This is done by picking $V = 1/(\sqrt{g}\chi)\tilde{n}$, where $\tilde{n}$ is an extension of the normal. Indeed, this verifies $V \cdot n = 1/(\sqrt{g}\chi)$, so that 
    \begin{align*}
        \div_{\Gamma}(B(\Omega) (V \cdot n)) &= \frac{1}{\sqrt{g}}\left(\partial _\phi \left(\sqrt{g} \frac{\chi}{\chi \sqrt{g}}\right)+\partial _\theta \left(\sqrt{g} \frac{\omega\chi}{\chi \sqrt{g}}\right)\right)
        \\
        &=0.
    \end{align*}
    As a consequence, the solution $u_V$ of \cref{eq:u} is constant, and by \cref{prop:derivatives_linearizable} $(X'_V)^\theta$ is given by 
    $$
    (X'_V)^\theta = \frac{2\mathbb{II}\left(B(\Omega),B(\Omega)^\perp\right)}{(\det{g})\chi^3}.
    $$
    Since we assume that $\mathbb{II}\left(B(\Omega), B(\Omega)^\perp\right)$ vanishes nowhere, it is either positive or negative on $\partial \Omega$. As a consequence 
    $$
    \theta \mapsto \int_{0}^{1} (X'_V)^\theta(\phi, \theta + \omega \phi) d\phi,
    $$
    is also either positive or negative on $S^1$, and has therefore a nonzero average.
\end{proof}

\begin{comment}
\begin{remark}
    Unsurprisingly, we observe that only the restriction of $V$ to $\partial \Omega$ is used in the proof. As such, we may always choose $V$ to be compactly supported in a neighborhood of $\partial \Omega$. \rouwarning{Encore nécessaire ?}
\end{remark}

\begin{remark}\label{rem:remark_II}
    From the proof final step of the proof of \cref{th:surjectivity_diophantine}, we see that the condition on $\mathbb{II}(B, B^\perp)$ is stronger than necessary. However, as we will see in the next section, it is possible to find a large family of toroidal domains which verify this condition.
\end{remark}
\end{comment}

{
\section{Conclusion and perspectives}
\paragraph{Conclusion}
In this paper, we have established a shape differentiability result for the Poincaré maps of harmonic fields, and studied some properties of the shape derivative in specific cases. The shape differentiability of the Poincaré map was obtained by proving shape differentiability of harmonic fields in the smooth category using suitable pullbacks on the variational spaces and elliptic regularity. We then studied the case of axisymmetric domains, for which we found that the shape derivative of the Poincaré map of harmonic fields always vanishes. After that, we have found that when the domain has a Poincaré map which is a diophantine rotation on the boundary, the shape derivative may be any smooth function of the circle under an additional assumption relating the curvature of the boundary and the harmonic field. 

\paragraph{Perspectives}
First, we saw in \cref{sec:axisymmetric} that the shape derivative of the Poincaré map of the harmonic field is always zero in the case of the standard torus. One could then naturally ask if the second order derivative may in this case produce changes in the Poincaré map, and in particular, changes in the rotation number. For this, one would need to find an expression for this second order shape derivative. Although the author has made some preliminary computations in this direction, it seems that the formulas we obtain are much more difficult to deal with that in the first order case so that coming up with deformations of the domain leading to changes in the Poincaré map at second order is not an easy task. 

Also, we believe that the main result of \cref{sec:diophantine} leads to two interesting questions which we leave open. First, it is still not clear to the author if the assumption relating the harmonic field and the second fundamental form in \cref{th:surjectivity_diophantine} is often realized in practice. The author attempted to study this condition in the case of the thin toroidal domains studied in \cite{enciso_existence_2015}, but it seems that this assumption is not verified in this case. In fact, we found that the approximate harmonic field denoted as $h_0$ in \cite[Section 5]{enciso_existence_2015} is exactly aligned with the lowest principal curvature direction of the boundary, and that taking some finer approximations leads to changes in sign for the quantity used in the assumption. Although this fact is interesting in itself, it means that the additional assumption used for \cref{th:surjectivity_diophantine} is not satisfied for thin toroidal domains (see \cref{rem:assumption}). One could then try to find other domains for which the assumption is satisfied, but this does not seem like an easy task. Also, we note that the proof of \cref{th:surjectivity_diophantine} mostly uses tangential deformations, which is unusual in shape differentiation. The author believes that proving surjectivity of the differential using only the normal component of the deformation in \cref{eq:X'_linearized} could potentially help to remove this additional condition on the second fundamental form. However, the dependence between $u_V$ and $V \cdot n$ in \cref{eq:X'_linearized} is highly non-explicit which complicates this approach. 

Finally, we believe that \cref{th:surjectivity_diophantine} may lead to interesting local properties of $\Pi$ around embeddings where it applies. Indeed, if we were working in the simpler case of Banach spaces, we would obtain that $\Pi$ is locally surjective around such domains. As a consequence, we would find that generic perturbations would lead to Morse--Smale diffeomorphisms of the circle, and thus rational rotation numbers. However, since we are working in Fréchet Spaces rather than Banach ones, it is necessary to prove surjectivity of the differential in a neighborhood of an embedding satisfying the conditions of \cref{th:surjectivity_diophantine} to prove the local surjectivity of $\Pi$. We refer the reader to \cite{hamilton_inverse_1982} for an example of a local surjectivity theorem in Fréchet spaces. However, the construction used for the proof of \cref{th:surjectivity_diophantine} is not applicable for embeddings close to the original one as it is highly dependent on the diophantine condition of the rotation number. Furthermore, this construction also uses cohomological equations, which leads to a loss of derivatives phenomenon. This implies that, although this approach proves surjectivity of the differential in the smooth category, we do not expect this result to hold for finite regularity. As a consequence, the author believes that one cannot reduce the regularity assumptions to obtain similar results in Banach spaces. The author therefore believes, once again, that another proof of \cref{th:surjectivity_diophantine} using the normal components of the deformation in \cref{eq:X'_linearized} could be beneficial to study the local properties of $\Pi$.
}

\begin{comment}
\appendix 
\section{Circle diffeomorphisms and chronological calculus}\label{app:appendix_cc}
\input{appendix_cc.tex}

\section{Influence of the choice of coordinates on the Poincaré map}
\input{coordinates.tex}
\end{comment}

{\paragraph*{Acknowledgements} The author would like to express his sincere gratitude to Andrei Agrachev for the insightful conversations which shaped the direction of this research project. Additionally, the author thanks Ugo Boscain and Mario Sigalotti for their advice throughout the development of the article, and Wadim Gerner for his geometrical insights on some aspects of the problem. This research received support from Inria AEX StellaCage.}

\bibliographystyle{alpha}
\bibliography{main}

\end{document}